\documentclass[a4paper,reqno,12pt]{amsart}
\usepackage{geometry}
\usepackage[dvips]{graphicx}
\usepackage[colorlinks=true,linkcolor=blue,citecolor=blue]{hyperref}

\usepackage{amsfonts}
\usepackage{bbm}
\numberwithin{equation}{section}
\usepackage{color}
\usepackage{graphicx}
\usepackage{indentfirst}
\usepackage{color}
\usepackage{amssymb}
\usepackage{mathrsfs}
\usepackage{xy}
\usepackage{tikz}

\xyoption{all}
 \def\Hom{\mbox{\rm Hom}}  \def\Iso{\mbox{\rm Iso}\,}
 \def\fin{\hfill$\square$}   \def\mod{\mbox{\rm \textbf{mod}}\,}
    
\def\cone{\mbox{\rm Cone}}\def\cocone{\mbox{\rm CoCone}}
\def\ind{\mbox{\rm ind}\,}\def\cim{\mbox{\rm sim}}\def\add{\mbox{\rm add}}
\def\A{\mathcal{A}\,} 
\def\Filt{\mbox{\rm \textbf{Filt}}}
\def\f-tors{\mbox{\rm f-tors}} \def\s\tau-tilt{\mbox{\rm s\tau-tilt}}\def\Fac{\mbox{\rm Fac}}
 
\def\Sub{\mbox{\rm Sub}}\def\tors{\mbox{\rm tors}}\def\torf{\mbox{\rm torf}}\def\id{\mbox{\rm id}}\def\tor{\mbox{\rm tor}}
\def\Im{\mbox{\rm Im}}

\def\C{\mathscr{C}}

\def\s{\mathfrak{s}}

\theoremstyle{plain}
\newtheorem{theorem}{\bf Theorem}[section]
\newtheorem{lemma}[theorem]{\bf Lemma}
\newtheorem{corollary}[theorem]{\bf Corollary}
\newtheorem{proposition}[theorem]{\bf Proposition}

\newtheorem*{theorem A}{\bf Theorem A}
\newtheorem*{theorem B}{\bf Theorem B}
\newtheorem*{theorem C}{\bf Theorem C}
\newtheorem*{theorem D}{\bf Theorem D}
\newtheorem*{theorem E}{\bf Theorem E}
\theoremstyle{definition}
\newtheorem{definition}[theorem]{\bf Definition}
\newtheorem{remark}[theorem]{\bf Remark}
\newtheorem{example}[theorem]{\bf Example}

\newcommand{\bt}{\begin{theorem}}
\newcommand{\et}{\end{theorem}}
\newcommand{\bl}{\begin{lemma}}
\newcommand{\el}{\end{lemma}}
\newcommand{\bd}{\begin{definition}}
\newcommand{\ed}{\end{definition}}
\newcommand{\bc}{\begin{corollary}}
\newcommand{\ec}{\end{corollary}}
\newcommand{\bp}{\begin{proof}}
\newcommand{\ep}{\end{proof}}
\newcommand{\bx}{\begin{example}}
\newcommand{\ex}{\end{example}}
\newcommand{\br}{\begin{remark}}
\newcommand{\er}{\end{remark}}
\newcommand{\be}{\begin{equation}}
\newcommand{\ee}{\end{equation}}
\newcommand{\ba}{\begin{align}}
\newcommand{\ea}{\end{align}}
\newcommand{\bn}{\begin{enumerate}}
\newcommand{\en}{\end{enumerate}}
\newcommand{\bcs}{\begin{cases}}
\newcommand{\ecs}{\end{cases}}

\newcommand{\RNum}[1]{\uppercase\expandafter{\romannumeral #1\relax}}

\makeatletter
\renewcommand{\section}{\@startsection{section}{1}{0mm}
  {-\baselineskip}{0.5\baselineskip}{\bf\leftline}}
\makeatother

\begin{document}

\title[Extriangulated length  categories: torsion classes and $\tau$-tilting theory]{Extriangulated length categories: torsion\\[2mm] classes and $\tau$-tilting theory}
\author[Wang, Wei, Zhang and Zhou]{Li Wang,  Jiaqun Wei, Haicheng Zhang and Panyue Zhou}
\address{School of Mathematics-Physics and Finance, Anhui Polytechnic University, 241000 Wuhu, Anhui, P. R. China \endgraf}
\email{wl04221995@163.com {\rm (L. Wang)}}
\address{School of Mathematical Sciences, Zhejiang Normal University, 321004 Jinhua, Zhejiang, P. R. China\endgraf}
\email{weijiaqun5479@zjnu.edu.cn {\rm (J. Wei)}}

\address{School of Mathematical Sciences, Nanjing Normal University,
210023 Nanjing, Jiangsu,  P. R.  China\endgraf}
\email{zhanghc@njnu.edu.cn {\rm (H. Zhang)}}

\address{School of Mathematics and Statistics, Changsha University of Science and Technology, 410114 Changsha, Hunan,  P. R. China\endgraf}
\email{panyuezhou@163.com {\rm (P. Zhou)}}


\begin{abstract}
This paper introduces the notion of extriangulated length categories, whose prototypical examples include abelian length categories and bounded derived categories of finite dimensional algebras with finite global dimension. We prove that an extriangulated category
$\mathcal{A}$ is a length category if and only if $\mathcal{A}$
admits a simple-minded system. Subsequently, we study the partially ordered set
$\tor_{\Theta}(\mathcal{A})$ of torsion classes in an extriangulated length category $(\mathcal{A},\Theta)$ from the perspective of lattice theory. It is shown that
$\tor_{\Theta}(\mathcal{A})$  forms a complete lattice, which is further proved to be completely semidistributive and algebraic. Moreover, we describe the arrows in the Hasse quiver of
$\tor_{\Theta}(\mathcal{A})$ using brick labeling. Finally, we introduce the concepts of support torsion classes
and support $\tau$-tilting subcategories in extriangulated length categories and  establish a bijection between these two notions, thereby generalizing the Adachi-Iyama-Reiten bijection for functorially finite torsion classes.
\end{abstract}

\maketitle

{\vspace{-1cm}
\tableofcontents}

\section{Introduction}

The concept of abelian categories, introduced and thoroughly explored by Grothendieck, serves as a unified  framework that elegantly generalizes both module categories over rings and categories of sheaves over schemes. Among these, abelian length categories \cite{Ga}, which are abelian categories with the length of every object finite, play a pivotal role in representation theory. Their importance stems from their structural richness and their applications in various mathematical domains, including
$\tau$-tilting theory \cite{En2}, lattice theory \cite{As}, and the study of strongly unbounded types \cite{Kr}. Schur's lemma underscores their significance by showing that the set of isomorphism classes of simple objects in an abelian length category forms a semibrick, which is a combinatorial object with rich mathematical properties.

In triangulated categories, the notion of a simple-minded system generalizes the role of simple objects in abelian length categories. This concept, introduced by Koenig and Liu \cite{KL} in the study of stable module categories, has been further investigated, particularly in relation to mutations \cite{Du}. Dugas \cite{Du} proved that bounded derived categories of finite dimensional algebras with finite global dimension admit simple-minded systems. This result allows us to interpret these derived categories as analogs of abelian length categories in the triangulated framework, thereby establishing a meaningful connection between the two contexts.

Recently, Nakaoka and Palu \cite{Na} proposed the framework of extriangulated categories, which synthesizes and extends properties from both triangulated and exact categories. Within this broader framework, it becomes natural to ask how the notion of abelian length categories can be extended. To address this, we introduce the concept of extriangulated length categories, which formalizes the idea of abelian length categories in the setting of extriangulated categories.

An extriangulated length category is defined as a pair
$(\mathcal{A},\Theta)$, where $\mathcal A$ is an extriangulated category and
$\Theta:\Iso(\mathcal{A})\rightarrow \mathbb{N}$ is a length function that ensures every morphism in
$\mathcal A$ satisfies the $\Theta$-admissibility condition (see Definition \ref{Def}). This definition preserves the essential characteristics of abelian length categories while adapting them to the more flexible extriangulated setting. Our first main result, stated below, establishes a foundational theorem for extriangulated length categories and opens the door to further theoretical exploration and applications.
\begin{theorem A}\text{\rm (see Theorems \ref{main1}, \ref{main-L} for details)} Let $\mathcal{A}$ be an  extriangulated category.

$(1)$ If $\mathcal{X}$ is a simple-minded system in $\mathcal{A}$, then  $(\mathcal{A},l_{\mathcal{X}})$ is an extriangulated length category, where $l_{\mathcal{X}}$ is  the minimal length of $\mathcal{X}$-filtrations.

$(2)$  $\mathcal{A}$ is an extriangulated length category if and only if $\mathcal{A}$ has a simple-minded system.
\end{theorem A}

An important source of examples for extriangulated length categories comes from the notion of \emph{length wide subcategories} (see Definition \ref{Def25}). These subcategories extend the concept of wide subcategories from abelian length categories to the extriangulated setting, thereby enriching the structural landscape of extriangulated categories. Analogous to their counterparts in abelian categories, length wide subcategories can be generated by semibricks, a combinatorial structure that captures key categorical properties.
This connection not only deepens our understanding of the role that semibricks play in categorical frameworks but also allows us to generalize classical results into the realm of extriangulated categories. In particular, we extend the celebrated result of Ringel \cite{Ri}, which describes the structure and generation of certain subcategories in abelian length categories. Our generalization shows that the underlying principles of Ringel's theorem hold in this broader context, offering new insights and potential applications. More specifically, we establish the following theorem, which serves as a cornerstone of our study.

\begin{theorem B} \text{\rm (see Theorem \ref{main-s} for details)} Let $\mathcal{A}$ be an extriangulated category. Then there exists a bijection between the following two sets.

$(1)$ The set of  semibricks $\mathcal{X}$ in $\mathcal{A}$.

$(2)$ The set of  length wide subcategories $(\mathcal{C},\Theta)$ in $\mathcal{A}$.

The mutually inverse maps are given by $\mathcal{X}\mapsto (\Filt_{\mathcal{A}}(\mathcal{X}),l_{\mathcal{X}})$ and $(\mathcal{C},\Theta)\mapsto \cim(\mathcal{C})$.
\end{theorem B}

The concept of torsion classes forms a cornerstone in the representation theory of algebras, has significant contributions and applications in areas such as
$\tau$-tilting theory \cite{Ad}, lattice theory \cite{De}, stability conditions \cite{Br}, and Hall algebras \cite{Tr}. These subcategories play a pivotal role in understanding and organizing the structural properties of various algebraic and categorical frameworks.
In this paper, we extend the notion of {torsion classes} to the setting of extriangulated length categories $(\mathcal{A},\Theta)$ and study their fundamental properties. Specifically, torsion classes in this context are subcategories closed under extensions and quotients (see Definition \ref{Torsion}), preserving key homological behaviors. These subcategories exhibit elegant and useful properties, as established in Lemma \ref{L-2-1} and Proposition \ref{P-2-6}.
Furthermore, we explore the collection
$\torf_{\Theta}(\mathcal{A})$ , which consists of all torsion classes in
$\mathcal A$, through the lens of lattice theory. This approach reveals structural insights and provides a unified framework that parallels the classical results for finite-dimensional algebras \cite{De}. In particular, we show that the set
$\torf_{\Theta}(\mathcal{A})$ forms a lattice with desirable combinatorial and categorical properties. These findings culminate in the following theorem, which generalizes known results and establishes a robust theoretical foundation for torsion classes in extriangulated length categories.

\begin{theorem C} \text{\rm (see Theorem \ref{main2} for details)}  Let $(\mathcal{A},\Theta)$ be an extriangulated length category.

$(1)$ The set $\tors_{\Theta}(\mathcal{A})$ is a complete lattice.

$(2)$  The complete lattice $\tors_{\Theta}(\mathcal{A})$ is  completely semidistributive and algebraic.
\end{theorem C}

For a given set $\mathcal X$, we denote its cardinality by
$\#\mathcal{X}$. Let
 $[\mathcal{U},\mathcal{T}]$
 be an interval in  the lattice $\tors_{\Theta}(\mathcal{A})$ of torsion classes in the extriangulated length category $(\mathcal A,\Theta)$. We denote by
 ${\rm brick}(\mathcal{U}^{\perp}\cap \mathcal{T})$ the set of isomorphism classes of bricks contained in
$\mathcal{U}^{\perp}\cap \mathcal{T}$, where
$\mathcal{U}^{\perp}$
  represents the right orthogonal of
$\mathcal U$ in $\mathcal A$. This framework provides a natural way to study the combinatorial and structural properties of intervals within the lattice
$\tors_{\Theta}(\mathcal{A})$.
Using this notation, we arrive at the following description of the intervals in
$\tors_{\Theta}(\mathcal{A})$, which connects the cardinalities of such intervals to the numbers of certain bricks. This characterization not only highlights the intrinsic link between torsion classes and bricks but also underscores the interplay between algebraic and combinatorial aspects of extriangulated length categories, offering new insights into their lattice-theoretic properties. The precise formulation of this description is presented in the following.

\begin{theorem D} \text{\rm (see Theorem \ref{main4} for details)}   Let $(\mathcal{A},\Theta)$ be an extriangulated length category. Suppose that $[\mathcal{U},\mathcal{T}]\subseteq\tors_{\Theta}(\mathcal{A})$.

$(1)$ We have $\#[\mathcal{U},\mathcal{T}]=1$ if and only if {\rm brick}$(\mathcal{U}^{\perp}\cap \mathcal{T})=\emptyset$.

$(2)$ If $\#${\rm brick}$(\mathcal{U}^{\perp}\cap \mathcal{T})=1$, then $\#[\mathcal{U},\mathcal{T}]=2$.

$(3)$ If $\#[\mathcal{U},\mathcal{T}]=2$, then there exists a brick $S\in{\rm brick }(\mathcal{U}^{\perp}\cap \mathcal{T})$ satisfying the following conditions:
\begin{itemize}
   \item [(i)]$\Theta(S)\leq\Theta(S')$ for any $S'\in{\rm brick }(\mathcal{U}^{\perp}\cap \mathcal{T})$.
   \item [(ii)] $\mathcal{T}=\mathrm{T}_{\Theta}(\mathcal{U}\cup  S)$ and $\mathcal{U}=\mathcal{T}\cap{^{\perp}}S$.
   \item [(iii)] $S$ satisfies $$S\in\bigcap_{S'\in{\rm brick }(\mathcal{U}^{\perp}\cap \mathcal{T})}(\Sub_{\Theta}(S')\cap\Fac_{\Theta}(S')).$$
  \end{itemize}
Moreover, the brick $S$ is unique up to isomorphism.
\end{theorem D}

Building on Theorem ${\rm  D}$,
we naturally extend the concept of brick labeling (see Definition \ref{brick})
to the Hasse quiver of
$\tors_{\Theta}(\mathcal{A})$, offering a generalization of the
framework introduced in \cite[Section 3.2]{De}.
This labeling provides a systematic way to assign bricks to the intervals in the lattice
$\torf_{\Theta}(\mathcal{A})$ of torsion-free classes, enriching our understanding of its combinatorial and categorical structure. Moreover, we characterize the arrows in the Hasse quiver of
$\tors_{\Theta}(\mathcal{A})$ using the brick labeling, as detailed in Proposition \ref{3-11}. This characterization deepens the connection between the morphisms in the quiver and the corresponding torsion classes.

Finally, we extend $\tau$-tilting theories to extriangulated length categories by leveraging the structure of torsion classes. To achieve this, we introduce the concepts of support torsion classes (see Definition \ref{d-5-2}) and support $\tau$-tilting subcategories (see Definition \ref{d-5-11}) in the context of extriangulated length categories. This notion serves as a natural analog of its counterpart in classical settings and enables us to generalize the celebrated bijection of Adachi-Iyama-Reiten (cf. \cite[Theorem 2.7]{Ad}) for functorially finite torsion classes. Our generalization highlights the versatility of $\tau$-tilting theory and its applicability in the broader framework of extriangulated categories. This development not only unifies and extends the existing results but also lays the groundwork for further exploration of the interplay between torsion classes,
$\tau$-tilting theories, and extriangulated length categories. The precise statement of our generalized bijection is presented as follows.

\begin{theorem E}  \text{\rm (see Theorem \ref{main5} for details)}  Let $(\mathcal{A}, \Theta)$ be an extriangulated length category with enough $\Theta$-projectives. Then we have a
bijection between the following sets.

$(1)$ The set of support torsion classes.

$(2)$ The set of support $\tau$-tilting subcategories.
\end{theorem E}

$\mathbf{Organization.}$ This paper is organized as follows. In Section 2, we summarize the fundamental definitions and key properties of extriangulated categories and filtration categories, and provide the foundational concepts required for subsequent discussions. Section 3 introduces the notion of extriangulated length categories and establishes several basic properties that will be utilized throughout the paper. In Section 4, we develop a lattice-theoretical framework for studying the partially ordered sets of torsion classes within extriangulated length categories, offering a structured approach to analyzing their relationships. Finally, in Section 5, we introduce the concepts of support torsion classes and support $\tau$-tilting subcategories. We also establish a bijection between these two notions, thereby extending
$\tau$-tilting theories to the framework of extriangulated categories.
\vspace{2mm}

$\mathbf{Conventions~and~Notation.}$ Throughout this paper, we assume that all considered categories are skeletally small and Krull-Schmidt, and that the subcategories are full and closed under isomorphisms. For an additive category $\mathscr{C}$, we denote by $\Iso(\mathscr{C})$ the set of isomorphism class of objects in $\mathscr{C}$. Given an object $X\in\mathscr{C}$, we denote by $\add(X)$ the full subcategory of $\mathscr{C}$ consisting of finite direct sums of direct summands of $X$.  For a finite dimensional algebra $\Lambda$ over a field $k$,  we denote by $\mod\Lambda$ the category of finite dimensional $\Lambda$-modules and by $D^{b}(\Lambda)$ the bounded derived category.

\section{Preliminaries}
In this section, we review some notations and properties of extriangulated categories and filtration subcategories, as introduced in \cite{Na} and \cite{Wa}, respectively.

\subsection{Extriangulated categories}
In this subsection,  we omit the detailed axioms and definitions on extriangulated categories, and only give terminologies and properties which we shall need later. For the precise definition and the detailed properties of extriangulated categories, we refer to \cite[Section 2]{Na}.

An {\em extriangulated category} $(\mathscr{C}, \mathbb{E},\mathfrak{s})$ is a triple
consisting of the following data and satisfying certain conditions:

(1)~ $\C$ is an additive category,
 and $\mathbb{E}\colon\C^{\rm op}\times \C\rightarrow {\rm Ab}$ is
 an additive bifunctor, where {\rm Ab} is the category of abelian groups.

(2)~ $\mathfrak{s}$ is a correspondence that assigns to each
 $\delta\in \mathbb{E}(Z,X)$ an equivalence class $[X \to Y \to Z]$  of complexes
in $\C$.

Here two complexes $X\xrightarrow{~f~}Y\xrightarrow{~g~}Z$ and
$X\xrightarrow{~f'~}Y\xrightarrow{~g'~}Z$ are \emph{equivalent} if there exists an isomorphism
$h\colon Y\to Y'$ such that the following diagram commutes:
$$\xymatrix{X\ar[r]^{f}\ar@{=}[d]&Y\ar[r]^{g}\ar[d]^{h}&Z\ar@{=}[d]\\
X\ar[r]^{f'}&Y'\ar[r]^{g'}&Z.}$$
Moreover, this triplet must satisfy the axioms
(ET1)-(ET4), (ET3)$^{\rm op}$ and (ET4)$^{\rm op}$ as detailed in \cite[Definition 2.12]{Na}.
In what follows, we often write $\C$ for an extriangulated category $(\C, \mathbb{E}, \mathfrak{s})$ for simplicity.

Given $\delta\in\mathbb{E}(C,A)$. If $\mathfrak{s}(\delta)=[A\stackrel{x}{\longrightarrow}B\stackrel{y}{\longrightarrow}C]$, then the sequence $A\stackrel{x}{\longrightarrow}B\stackrel{y}{\longrightarrow}C$ is called a {\em conflation}, $x$ is called an {\em inflation} and $y$ is called a {\em deflation}. In this case, we call $$A\stackrel{x}{\longrightarrow}B\stackrel{y}{\longrightarrow}C\stackrel{\delta}\dashrightarrow$$
is an $\mathbb{E}$-triangle, and  denote $A=\cocone(y)$ and $C=\cone(x)$. For two subcategories $\mathcal{T},\mathcal{F}\subseteq\mathscr{C}$, we define
$$\mathcal{T}^{\perp}=\{M\in\mathscr{C}~|~\Hom(X,M)=0~\text{for any}~X\in\mathcal{T}\},$$
$$^{\perp}\mathcal{T}=\{M\in\mathscr{C}~|~\Hom(M,X)=0~\text{for any}~X\in\mathcal{T}\},$$
$$\mathcal{T}\ast \mathcal{F}=\{M\in \mathscr{C}~|~\text{there exists an $\mathbb{E}$-triangle}~T\stackrel{}{\longrightarrow}M\stackrel{}{\longrightarrow}F\stackrel{}\dashrightarrow~\text{with}~T\in\mathcal{T},F\in\mathcal{F} \}.$$

We say that a subcategory $\mathcal{T}\subseteq \C$ is {\em extension-closed} if $\mathcal{T}\ast \mathcal{T}\subseteq\mathcal{T}$.
In this case, $\mathcal{T}$ is an extriangulated category with extriangulated structure induced from that of $\mathscr C$ (see \cite[Remark 2.18]{Na}).

\begin{remark}
We know that both exact categories and triangulated categories are examples of extriangulated categories \cite[Example 2.13]{Na}, and that extension-closed subcategories of extriangulated categories are also extriangulated \cite[Remark 2.18]{Na}. Furthermore, there exist extriangulated categories that are neither exact nor triangulated, as demonstrated in \cite{FHZZ,  HZZ, Na, NP1, ZZ, ZhZ}.
\end{remark}

\subsection{Filtration subcategories}
In this subsection, we always assume that $\C=(\mathscr{C}, \mathbb{E},\mathfrak{s})$ is an  extriangulated category. Let $\mathcal{X}$ be a collection of objects in $\mathscr{C}$. The {\em filtration subcategory} $\mathbf{Filt_{\mathscr{C}}(\mathcal{X})}$ is consisting of all objects $M$ admitting a finite filtration of the form
\begin{equation*}
0=M_{0}\stackrel{f_{0}}{\longrightarrow}M_{1}\stackrel{f_{1}}{\longrightarrow}M_{2}{\longrightarrow}\cdots\xrightarrow{f_{n-1}}M_{n}=M
\end{equation*}
with $f_{i}$ being an inflation and $\cone(f_{i})\in\mathcal{X}$ for any $0\leq i\leq n-1$. For each object $M\in\Filt_{\mathscr{C}}(\mathcal{X})$, the minimal length of $\mathcal{X}$-{filtrations} of $M$ is called the $\mathcal{X}$-{\em length} of $M$, which is denoted by $l_{\mathcal{X}}(M)$.

The following lemma can be found in \cite{Wa}, which is useful in the sequel.

\begin{proposition}\label{p-1-1} Let $\mathcal{X}$ be a collection of objects in $\mathscr{C}$ and set $\mathcal{A}:=\Filt_{\mathscr{C}}(\mathcal{X})$.

$(1)$ For any $\mathbb{E}$-triangle $A\stackrel{}{\longrightarrow}B\stackrel{}{\longrightarrow}C\stackrel{}\dashrightarrow$ in $\mathcal{A}$, we have $l_{\mathcal{X}}(B)\leq l_{\mathcal{X}}(A)+l_{\mathcal{X}}(C)$.

$(2)$ $\mathcal{A}$ is the smallest extension-closed subcategory containing $\mathcal{X}$ in $\mathscr{C}$.

$(3)$ For any $M\in \mathcal{A}$, there exist two $\mathbb{E}$-triangles $X_{1}\stackrel{}{\longrightarrow}M\stackrel{}{\longrightarrow}M'\stackrel{}\dashrightarrow$ and $M''\stackrel{}{\longrightarrow}M\stackrel{}{\longrightarrow}X_{2}\stackrel{}\dashrightarrow$ in $\mathcal{A}$ such that $X_{1},X_{2}\in \mathcal{X}$ and $l_{\mathcal{X}}(M')=l_{\mathcal{X}}(M'')=l_{\mathcal{X}}(M)-1$.
\end{proposition}

\begin{proof}
The statements can be found as follows: (1) in \cite[Lemma 2.7]{Wa}, (2) in \cite[Lemma 2.8]{Wa}, and (3) in \cite[Lemma 2.9]{Wa}.
\end{proof}

Recall that an object $M\in\mathscr{C}$ is called a {\em brick}
 if its endomorphism ring is a division ring.
 A set $\mathcal{X}$ of isomorphism classes of bricks in $\mathscr{C}$
 is called a {\em semibrick} if $\Hom_{\mathscr{C}}(X_1,X_2)=0$
  for any two non-isomorphic objects $X_1,X_2$ in $\mathcal{X}$.
A semibrick $\mathcal{X}$ is said to be {\em proper} if for any $X\in\mathcal{X}$,
there does not exist an $\mathbb{E}$-triangle of the form $X\stackrel{}{\longrightarrow}0\stackrel{}{\longrightarrow}N\stackrel{}\dashrightarrow$ with $N\in\Filt_{\mathcal{A}}(\mathcal{X})$.
It is straightforward to verify that the notions of semibricks and proper semibricks coincide in exact categories. However, they differ in triangulated categories (see \cite[Example 4.2]{Wa2}). We summarize the following basic properties.

\begin{proposition}\label{p-1-2}  Let $\mathcal{X}$ be a semibrick in $\mathscr{C}$ and set $\mathcal{A}:=\Filt_{\mathscr{C}}(\mathcal{X})$.

$(1)$ $\mathcal{A}$ is closed under direct summands.

$(2)$ If $A=B\oplus C\in\mathcal{A}$, then $l_{\mathcal{X}}(A)=l_{\mathcal{X}}(B)+l_{\mathcal{X}}(C)$.

$(3)$ If $f:X\rightarrow M$ is a non-zero morphism in $\mathcal{A}$ with $X\in\mathcal{X}$, then $f$ is an inflation and $l_{\mathcal{X}}(\cone(f))=l_{\mathcal{X}}(M)-1$. Dually,  if $g:M\rightarrow X$ is a non-zero morphism in $\mathcal{A}$ with $X\in\mathcal{X}$, then $g$ is a deflation and $l_{\mathcal{X}}(\cocone(g))=l_{\mathcal{X}}(M)-1$.

$(4)$  $\mathcal{X}$ is a proper semibrick if and only if $l_{\mathcal{X}}(B)=l_{\mathcal{X}}(A)+l_{\mathcal{X}}(C)$ for any $\mathbb{E}$-triangle $A\stackrel{}{\longrightarrow}B\stackrel{}{\longrightarrow}C\stackrel{}\dashrightarrow$ in $\mathcal{A}$.
\end{proposition}

\begin{proof} The first two statements can be found in \cite[Lemma 5.4]{Wa},  (3) in \cite[Lemma 3.5 and Corollary 3.6]{Wa} and (4) in  \cite[Lemma 3.9]{Wa2}.
\end{proof}

\begin{remark}\label{r-1-3}
Let $\mathcal{X}$ be a semibrick in $\mathscr{C}$ and set $\mathcal{A}:=\Filt_{\mathscr{C}}(\mathcal{X})$. By Proposition \ref{p-1-1} and Proposition \ref{p-1-2}, any object $M\in\mathcal{A}$ has at least one non-zero inflation $X\rightarrow M$ and one non-zero deflation $M\rightarrow X'$ with $l_{\mathcal{X}}(X)=l_{\mathcal{X}}(X')=1$.
\end{remark}

\section{Extriangulated length categories}
In this section, we introduce the concept of extriangulated length categories and establish their fundamental properties. We prove that extriangulated length categories correspond precisely to
those categories arising from simple-minded systems. Notably, many significant categories in representation theory fall within the framework of extriangulated length categories, highlighting their broad applicability and importance. Throughout the remainder of this section, we fix an extriangulated category
$(\mathcal{A}, \mathbb{E},\mathfrak{s})$ as the basis for our discussion.

 \subsection{Definitions and basic properties}

  We start with introducing the following notations.

\begin{definition}\label{length} We say that a map $\Theta:\Iso(\mathcal{A})\rightarrow \mathbb{N}$ is a {\em length function} on $\mathcal{A}$ if it satisfies the following conditions:
\begin{itemize}
   \item [(1)] $\Theta(X)=0$ if and only if $X\cong0$.
   \item [(2)] For any $\mathbb{E}$-triangle $X\stackrel{}{\longrightarrow}L\stackrel{}{\longrightarrow}M\stackrel{\delta}\dashrightarrow$ in $\mathcal{A}$, we have $\Theta(L)\leq \Theta(X)+\Theta(M)$. In addition, if $\delta=0$, then $\Theta(L)=\Theta(X)+\Theta(M)$.
  \end{itemize}
A length function $\Theta$ is said to {\em stable} if
it satisfies the condition $\Theta(L)=\Theta(X)+\Theta(M)$ for any $\mathbb{E}$-triangle $X\stackrel{}{\longrightarrow}L\stackrel{}{\longrightarrow}M\stackrel{}\dashrightarrow$ in $\mathcal{A}$.
\end{definition}

 We provide an elementary example of length functions.

\begin{example}Let $\mathcal{C}$ be a Krull-Schmidt category. One can regard $\mathcal{C}$ as an exact category with a split exact structure, i.e., a conflation in   $\mathcal{C}$ is a kernel-cokernel pair of the form $A\rightarrowtail A\oplus B\twoheadrightarrow B$. Then we have a stable length function $\Theta:\Iso(\mathcal{C})\rightarrow \mathbb{N}$ such that $\Theta(M)=|M|$ for any $M\in \mathcal{C}$, where $|M|$ denotes the number of non-isomorphic indecomposable direct summands of $M$.
\end{example}

 A semibrick $\mathcal{X}$ in $\mathcal{A}$ is called a {\em simple-minded system}  if $\mathcal{A}=\Filt_{\mathcal{A}}(\mathcal{X})$ (see \cite[Definition 5.10]{Wa}).
  Simple-minded systems also serve as a source of length functions, as demonstrated by the following lemma.

\begin{lemma}\label{WL} Let $\mathcal{X}$ be a simple-minded system in $\mathcal{A}$. Then $l_{\mathcal{X}}$ is a length function on $\mathcal{A}$. In addition, $l_{\mathcal{X}}$ is stable if and only if $\mathcal{X}$ is proper.
\end{lemma}
\begin{proof} This immediately follows from Proposition \ref{p-1-1} and Proposition \ref{p-1-2}.
\end{proof}

\begin{definition}\label{Length}  Let $\Theta$ be a length function on $\mathcal{A}$. We say that an $\mathbb{E}$-triangle $$X\stackrel{x}{\longrightarrow}L\stackrel{y}{\longrightarrow}M\stackrel{\delta}\dashrightarrow$$ in $\mathcal{A}$ is {\em $\Theta$-stable} (or {\em stable} when $\Theta$ is clear from the context) if $\Theta(L)=\Theta(X)+\Theta(M)$. In this case,  $x$ is called a {\em $\Theta$-inflation}, $y$ is called a {\em $\Theta$-deflation} and $\delta$ is called a  {\em $\Theta$-extension}. If there is no confusion, we depict $x$ by the symbol $X\rightarrowtail L$ and $y$ by $L\twoheadrightarrow M$. 
\end{definition}

We have the following useful characterizations of stable $\mathbb{E}$-triangles.
\begin{lemma}\label{L2-4}  Let  $\Theta$  be a  length function on $\mathcal{A}$. For an $\mathbb{E}$-triangle $$\eta:~~X\stackrel{f}{\longrightarrow}L\stackrel{g}{\longrightarrow}M\dashrightarrow$$
 in $\mathcal{A}$,  the following statements hold:

$(1)$ $f$ is a $\Theta$-inflation if and only if $g$ is a $\Theta$-deflation.

$(2)$ If $\eta$ is stable, then $f=0$ if and only if $X\cong 0$.

$(2')$  If $\eta$ is stable, then $g=0$ if and only if $M\cong 0$.

$(3)$ If $\eta$ is stable, then $f$ is an isomorphism if and only if $\Theta(M)=0$.

$(3')$  If $\eta$ is stable, then $g$ is an isomorphism if and only if $\Theta(X)=0$.
\end{lemma}

\begin{proof} These follow from Definition \ref{Length} and \cite[Corollary 3.5]{Na}.
\end{proof}

\begin{lemma}\label{L-1}Let  $\Theta$  be a length function on $\mathcal{A}$. Consider the following commutative diagrams of $\mathbb{E}$-triangles:
\begin{equation}\label{COM}
\begin{array}{l}
\xymatrix{
 A \ar@{=}[d] \ar[r]^-{f} & B  \ar[d]^-{g} \ar[r]^-{} & C  \ar[d] \ar@{-->}[r]^-{}  &  \\
  A\ar[r]^-{} & E\ar[d]^{} \ar[r]^{i} &  D\ar[d]^{j} \ar@{-->}[r]^-{}  &  \\
  &  F \ar@{-->}[d]^-{}  \ar@{=}[r]^{} &F\ar@{-->}[d]^-{} \\
  && }
\end{array}
\end{equation}

$(1)$ If $f$ and $g$ are $\Theta$-inflations, then each  $\mathbb{E}$-triangle in {\rm (\ref{COM})} is stable.

$(2)$ If $i$ and $j$ are $\Theta$-deflations, then each  $\mathbb{E}$-triangle in {\rm(\ref{COM})} is  stable.
\end{lemma}
\begin{proof}  We  only prove the first statement and the second one can be shown similarly. First of all, by Lemma \ref{L2-4}, $\Theta(A)+\Theta(C)= \Theta(B)$ and $\Theta(B)+\Theta(F)= \Theta(E)$. To see that $A\stackrel{}{\longrightarrow}E\stackrel{}{\longrightarrow}D\dashrightarrow$  is stable, we calculate:
\begin{flalign*} \Theta(E) &=\Theta(B)+\Theta(F)=\Theta(A)+\Theta(C)+\Theta(F)\geq\Theta(A)+\Theta(D)\geq\Theta(E).
\end{flalign*}
The preceding argument gives $\Theta(A)+\Theta(D)=\Theta(E)$.  A similar calculation shows that $\Theta(C)+\Theta(F)= \Theta(D)$.
\end{proof}

\begin{definition}\label{L-2-6} Let  $\Theta$  be a  length function on $\mathcal{A}$. A morphism $f:M\rightarrow N$ in $\mathcal{A}$ is called {\em $\Theta$-admissible} if $f$ admits a {\em $\Theta$-decomposition $(i_{f}, X_{f},j_{f})$}, i.e. there is a commutative diagram
 $$\xymatrix{
   M\ar[rr]^{f}  \ar@{->>}[dr]_{i_{f}}
                &  &   N     \\
                & X_{f}  \ar@{>->}[ur]_{j_{f}}                }
$$
such that $i_{f}$ is a $\Theta$-deflation and $j_{f}$ is a $\Theta$-inflation.
\end{definition}

This motivates the following definition.

\begin{definition}\label{Def} Let  $\Theta$  be a  length function on $\mathcal{A}$. We say that $((\mathcal{A}, \mathbb{E},\mathfrak{s}),\Theta)$ is an  {\em  extriangulated length category},  or simply a {\em length category}, if every morphism in $\mathcal{A}$ is $\Theta$-admissible.
For simplicity, we often write  $(\mathcal{A},\Theta)$ for  $((\mathcal{A}, \mathbb{E},\mathfrak{s}),\Theta)$ when $\mathbb{E}$ and $\mathfrak{s}$ are clear from the context.
\end{definition}
We will see later in Section 3.5 that abelian length categories and bounded derived categories of finite dimensional algebras with finite global dimension are  length categories. The following theorem provides a method to obtaining length categories.

\begin{theorem}\label{main1} Let $\mathcal{X}$ be a simple-minded system in $\mathcal{A}$. Then $(\mathcal{A},l_{\mathcal{X}})$ is a length category.
\end{theorem}

\begin{proof} By Lemma \ref{WL}, it suffices to show that each morphism in $\Filt_{\mathcal{A}}(\mathcal{X})$ is $l_{\mathcal{X}}$-admissible. To prove this, we take a morphism $f:M\rightarrow N$ in  $\Filt_{\mathcal{A}}(\mathcal{X})$. The case of $f=0$ is trivial. Hence we always assume that $f$ is a nonzero morphism. We proceed the proof by induction on $l_{\mathcal{X}}(M)+l_{\mathcal{X}}(N)$.

The case of $l_{\mathcal{X}}(M)=1$ or $l_{\mathcal{X}}(N)=1$ follows from Proposition \ref{p-1-2}. Now we consider the general case.
By Proposition \ref{p-1-1}, we have an $l_{\mathcal{X}}$-stable $\mathbb{E}$-triangle of the form
$$X\stackrel{a}\rightarrowtail M\stackrel{b}\twoheadrightarrow \cone(a)\stackrel{}\dashrightarrow$$
with $X\in \mathcal{X}$ and $l_{\mathcal{X}}(\cone(a))=l_{\mathcal{X}}(M)-1$. We consider two cases:

$(\mathbf{Case~1})$  $fa=0$. In this case, there exists a morphism $b':\cone(a)\rightarrow N$ such that $f=b'b$, i.e., we have a commutative diagram
\begin{equation*}
\xymatrix{
   M\ar[rr]^{f}  \ar@{->>}[dr]_{b}
                &  &   N     \\
                & \cone(a)  \ar[ur]_{b'}                }
\end{equation*}
Note that $b'\neq0$ and $l_{\mathcal{X}}(\cone(a))=l_{\mathcal{X}}(M)-1$. By induction hypothesis, there exists
an $l_{\mathcal{X}}$-decomposition $(i_{a'},X_{b'},j_{b'})$ of $b'$. Thus $(i_{a'}b,X_{b'},j_{b'})$ is an $l_{\mathcal{X}}$-decomposition of $f$ by Lemma \ref{L-1}.

$(\mathbf{Case~2})$  $fa\neq0$. In this case, by Proposition \ref{p-1-2}, we have a commutative diagram
\begin{equation}\label{10-24}
\begin{array}{l}
\xymatrix{
  X~\ar@{=}[d] \ar@{>->}[r]^-{a} & M\ar[d]_-{f} \ar@{->>}[r]^-{b} & \cone(a) \ar[d]_-{t} \ar@{-->}[r]^-{}  &  \\
  X~ \ar@{>->}[r]^-{fa} & N \ar@{->>}[r]^-{c} & \cone(fa) \ar@{-->}[r]^-{}&  }
\end{array}
\end{equation}
with $X\in \mathcal{X}$ and $l_{\mathcal{X}}(\cone(fa))=l_{\mathcal{X}}(N)-1$. If $t=0$,  there exists a morphism $f':M\rightarrow X$ such that $f=faf'$. Note that $X\in \mathcal{X}$ and $f'\neq0$. Thus  $f'$ is an
$l_{\mathcal{X}}$-deflation and thus $(f',X,fa)$ is an $l_{\mathcal{X}}$-decomposition of $f$.

Next, we focus on the case of $t\neq0$. Since $l_{\mathcal{X}}(\cone(a))+l_{\mathcal{X}}(\cone(fa))<l_{\mathcal{X}}(M)+l_{\mathcal{X}}(N)$, by induction hypothesis, there exists an $l_{\mathcal{X}}$-decomposition
as follows:
\begin{equation*}
\xymatrix{
   \cone(a)\ar[rr]^{t}  \ar@{->>}[dr]_{e_{1}}
                &  &   \cone(fa)     \\
                & T  \ar@{>->}[ur]_-{e_{2}}                }
\end{equation*}
Applying $\rm (ET4)^{op}$ yields the following  commutative diagram
\begin{equation}\label{3.14}
\begin{array}{l}
\xymatrix{
    M \ar@/_/[ddr]_{f} \ar@/^/[drr]^{e_{1}b}
    \ar@{.>}[dr]|-{l}  &  &   &  \\
  X~ \ar@{=}[d] \ar@{>->}[r] &M' \vphantom{\big|} \ar@{>->}[d]_-{f'}
  \ar@{->>}[r]^-{h'} & T\vphantom{\big|}  \ar@{>->}[d]_-{e_{2}}  &  \\
  X~ \ar@{>->}[r]^-{fa} & N\ar@{->>}[r]^-{c}\ar@{->>}[d]^-{} & \cone(fa)\ar@{->>}[d]^-{}  &  \\
 & \cone(e_{2})\ar@{=}[r] & \cone(e_{2}). & }
\end{array}
\end{equation}
Note that each conflation in (\ref{3.14}) is stable by Lemma \ref{L-1}. On the other hand, since $e_{2}e_{1}b=tb=cf$, there exists a morphism $l:M\rightarrow M'$ such that $f=f'l$ and $e_{1}b=h'l$.
Clearly $l_{\mathcal{X}}(M')=1+l_{\mathcal{X}}(T)\leq 1+\cone(fa)=l_{\mathcal{X}}(N)$.  If $l_{\mathcal{X}}(M')<l_{\mathcal{X}}(N)$, then $l_{\mathcal{X}}(M)+l_{\mathcal{X}}(M')<l_{\mathcal{X}}(M)+l_{\mathcal{X}}(N)$. By  induction hypothesis, this assertion holds for $l$ and so is $f$.

Suppose that $l_{\mathcal{X}}(M')=l_{\mathcal{X}}(N)$. Then $l_{\mathcal{X}}(T)=l_{\mathcal{X}}(\cone(fa))$ and thus $e_{2}$ is an isomorphism. In this case, $t\cong e_{1}$ is an $l_{\mathcal{X}}$-deflation and thus $$l_{\mathcal{X}}(N)=1+\cone(fa)\leq1+\cone(a)=l_{\mathcal{X}}(M).$$
If $l_{\mathcal{X}}(N)= l_{\mathcal{X}}(M)$, then $l_{\mathcal{X}}(\cone(fa))= l_{\mathcal{X}}(\cone(a))$ and thus $t$ is a isomorphism. Applying \cite[Corollary 3.6]{Na} to the diagram (\ref{10-24}), we deduce that $f$ is an isomorphism. It is easy to see that $f$ is $l_{\mathcal{X}}$-admissible. It remains to prove the case of $l_{\mathcal{X}}(N)<l_{\mathcal{X}}(M)$. Since $t\cong e_{1}$ is an $l_{\mathcal{X}}$-deflation, so is $tb=cf$. Consider the following commutative diagram
\begin{equation*}
\xymatrix{
  \cocone(cf)~\ar@{.>}[d]^{s} \ar@{>->}[r]^-{m} & M\ar[d]_-{f} \ar@{->>}[r]^-{cf} & \cone(fa) \ar@{=}[d] \ar@{-->}[r]^-{}  &  \\
  X~ \ar@{>->}[r]^-{fa} & N \ar@{->>}[r]^-{c} & \cone(fa) \ar@{-->}[r]^-{\theta}&.  }
\end{equation*}
If $s=0$, then there exists a morphism $k:\cone(fa)\rightarrow N$ such that $f=kcf$. Since $l_{\mathcal{X}}(\cone(fa))<l_{\mathcal{X}}(M)$, by induction hypothesis, $k$ is  $l_{\mathcal{X}}$-admissible and so is $f$. If $s\neq0$, then $s$ is an $l_{\mathcal{X}}$-deflation.  Applying $\rm (ET4)$ yields the following  commutative diagram
\begin{equation*}
\xymatrix{
 \cocone(s)~ \ar@{=}[d] \ar@{>->}[r] &  \cocone(cf) \vphantom{\big|} \ar@{>->}[d]^-{m} \ar@{->>}[r]^-{s} & X\vphantom{\big|}   \ar@{>->}[d]^{d} \\
  \cocone(s)~\ar@{>->}[r] & M\ar@{->>}[d]^{cf} \ar@{->>}[r]^{t} &  N' \ar@{->>}[d] \\
  &  \cone(fa)  \ar@{=}[r] & \cone(fa). }
\end{equation*}
Since $fas=fm$, there exists a morphism $q:N'\rightarrow N$ such that $f=qt$ and $fa=qd$. Note that $l_{\mathcal{X}}(N')=1+l_{\mathcal{X}}(\cone(fa))=l_{\mathcal{X}}(N)<l_{\mathcal{X}}(M)$. Since $$l_{\mathcal{X}}(N')+l_{\mathcal{X}}(N)<l_{\mathcal{X}}(N)+l_{\mathcal{X}}(M),$$ by the induction hypothesis, this assertion holds for $q'$ and so is $f$.
\end{proof}

\begin{example}\label{E-3-10} The most elementary example in representation theory is the finitely generated module category $\mod \Lambda$ for a finite dimensional algebra $\Lambda$ over an algebraically closed field $k$. Denote by $\mathcal{S}$ the set of isomorphism classes of simple $\Lambda$-modules, which is a simple-minded system in $\mod \Lambda$. By using Theorem \ref{main1}, we know that $(\mod \Lambda,l_{\mathcal{S}})$ is a stable length category. For a $\Lambda$-module $M$, $l_{\mathcal{S}}(M)$ is just the length of composition series of $M$. It is clear that the same argument holds for abelian length categories. For more examples of length categories, see Section 3.5.
\end{example}

\subsection{Extriangulated length categories and simple-minded systems} For a
 length category  $(\mathcal{A},\Theta)$, we  define
$$\Theta_{1}:=\{M\in\Iso(\mathcal{A})~|~\Theta(M)\leq \Theta(N)~\text{for any}~N\in\mathcal{A}\}.$$
Without loss of generality, we always assume that $\Theta(M)=1$ for any $M\in\Theta_{1}$.  For $n\geq2$, we  inductively define various sets  as follows:
$$\Theta'_{n}=\{M\in \mathcal{A}~|~M\in\Theta_{n-1}^{\perp}\bigcap{^{\perp}}\Theta_{n-1},\Theta(M)=n\}~\text{and}~\Theta_{n}=\Theta_{n-1}\bigcup\Theta'_{n}.$$
Put $\Theta_{\infty}=\bigcup_{n\geq1}\Theta_{n}$. We have the following basic observation.

\begin{lemma}\label{L-3-77}  Let $(\mathcal{A},\Theta)$ be a  length category.  Suppose that $f:M\rightarrow N$ is a non-zero morphism in $\mathcal{A}$. Consider the following commutative diagram
$$
\xymatrix{
  P \ar@{->>}[r]^{h} & M\ar@{->>}[dr]_{i_{f}} \ar[rr]^{f} & &N~ \ar@{>->}[r]^{g} & Q\\
   &   &  X_{f}.\ar@{>->}[ur]_{j_{f}} &  &  }
$$
Set $\Theta(P)=p$, $\Theta(M)=m$, $\Theta(N)=n$, $\Theta(Q)=q$ and $\Theta(X_{f})=s$.

$(1)$ If $\Theta(M)=1$, then $f$ is a $\Theta$-inflation. If $\Theta(N)=1$, then $f$ is a $\Theta$-deflation. In particular, $\Theta_{1}$ is a semibrick.

$(2)$ Suppose that $s<m$. If $P\in \Theta'_{p}$, then  $2\leq s< m$ and  $X_{f}\notin \Theta'_{s}$.

$(2')$ Suppose that $s<n$. If $Q\in \Theta'_{q}$, then  $2\leq s< n$ and  $X_{f}\notin \Theta'_{s}$.

$(3)$ Suppose that $s<\min\{m,n\}$. The following statements hold.

{\rm(i)}~If $P\in \Theta'_{p}$, then $Q\notin \Theta'_{q}$.
{\rm(ii)}~If $Q\in \Theta'_{q}$, then $P\notin \Theta'_{p}$.
\end{lemma}

\begin{proof} Clearly $1\leq s\leq m\leq p$ and $1\leq s\leq n\leq q$.

 $(1)$ If $\Theta(M)=1$, then $\Theta(X_{f})=1$. Note that  Lemma \ref{L2-4} implies that $f\cong j_{f}$ is a $\Theta$-inflation. Dually, if $\Theta(N)=1$, then $f\cong i_{f}$ is a $\Theta$-deflation. Again by  Lemma \ref{L2-4}, this implies $\Theta_{1}$ is a semibrick.

$(2)$ Firstly, note that $i_{f}h$ is a $\Theta$-deflation by Lemma \ref{L-1}. Since $s\geq1$, we have $i_{f}h\neq0$  by Lemma \ref{L2-4}.
If $s=1$, then $X_{f}\in\Theta_{1}$ and thus $i_{f}h=0$, this is a  contradiction. Hence, we get $s>1$. If $X_{f}\in \Theta'_{s}$, then $X_{f}\in\Theta_{s}\subseteq\Theta_{p-1}$. Since $P\in{^{\perp}}\Theta_{p-1}$, we obtain $i_{f}h=0$, a contradiction.
 Similarly, we can prove $(2')$.

 $(3)$  Suppose that $P\in \Theta'_{p}$ or $Q\in \Theta'_{q}$. By (2) and $(2')$, we have $2\leq s<\min\{m,n\}$ and  $X_{f}\notin \Theta'_{s}$. We have two cases:

$(\mathbf{Case~1})$. There exists a non-zero morphism $a:X\rightarrow X_{f}$ such that $X\in\Theta_{s-1}$.

$(\mathbf{Case~2})$.  There exists a non-zero morphism  $a:X_{f}\rightarrow X$ such that $X\in\Theta_{s-1}$.

In both cases, we take a $\Theta$-decomposition $(i_{a},X_{a},j_{a})$ of $a$ with $\Theta(X_{a})=t$. Clearly $t\leq\Theta(X)\leq s-1$. In particular, either $X\in\Theta_{1}$ or $X\in\Theta'_{\Theta(X)}$.

We divide the proof into two steps.

 $\underline{\mathbf{Step~1}.}$ Suppose that $P\in \Theta'_{3}$. We claim that $Q\notin \Theta'_{q}$. Indeed, we have $s=2$ and then $X\in\Theta_{1}$. In (Case 1), by (1),  $a$ is a $\Theta$-inflation and so is $gj_{f}a$. If $Q\in \Theta'_{q}$, then $gj_{f}a=0$ and thus $X\cong0$, this is a contradiction.
  In (Case 2),  $a$ is a $\Theta$-deflation and so is $ai_{f}h$. Since $P\in \Theta'_{3}$, we get $ai_{f}h=0$, this is a contradiction. Similarly,  if $Q\in \Theta'_{3}$, then $P\notin \Theta'_{p}$.

 $\underline{\mathbf{Step~2}.}$ We proceed the proof of (3) by  induction on $p+q$. It suffices to consider the case of $p+q\geq6$ with $p\geq3$ and $q\geq3$. The case of $p+q=6$ follows from Step 2. Now consider the general case.  Assume that  $P\in \Theta'_{p}$. In (Case 1), we obtain the following diagram:
 $$
\xymatrix{
  X \ar@{=}[r]^{1} & X\ar@{->>}[dr]_{i_{a}} \ar[rr]^{a} & &X_{f}~  \ar@{>->}[r]^{gj_{f}} & Q\\
   &   &  X_{a}.\ar@{>->}[ur]_{j_{a}} &  &  }
$$
If $t=1$, then $X_{a}\in\Theta(X)$. In this case, if $Q\in \Theta'_{q}$, then $gj_{f}j_{a}=0$, this is a  contradiction. Assume that $t>1$. If $\Theta(X)=\Theta(X_{a})$, then $X\cong X_{a}$ and hence $gj_{f}a$ is a $\Theta$-inflation. Recall that $X\in\Theta'_{\Theta(X)}$ and $\Theta(X)\leq s-1< q$. If $Q\in \Theta'_{q}$, then $gj_{f}a=0$,  this is a  contradiction. Suppose that $\Theta(X_{a})<\Theta(X)$. Since $\Theta(X)+\Theta(Q)\leq s-1+q<p+q$, by induction, we get $Q\notin\Theta'_{q}$. In (Case 2), we obtain the following diagram:
$$
\xymatrix{
 P\ar@{->>}[r]^{i_{f}h} & X_{f}\ar@{->>}[dr]_{i_{a}} \ar[rr]^{a} & &X  \ar@{=}[r]^{1} & X\\
   &   &  X_{a}\ar@{>->}[ur]_{j_{a}} &  &  }
$$
If $t=1$, then $i_{a}i_{f}h=0$, this is a  contradiction. If $\Theta(X_{a})=\Theta(X)$, then $a\cong i_{a}$ is a $\Theta$-deflation and so is $ai_{f}h$.  This is a contradiction, since $P\in \Theta'_{p}$. Hence, we have
$1<t\leq\Theta(X_{f})\leq s-1$. Since $\Theta(P)+\Theta(X)\leq p+s-1<p+q$, by induction, we get $X\notin\Theta'_{\Theta(X)}$. This is a contradiction. Using the analogous arguments  used as above, we can prove  $P\notin \Theta'_{p}$ if  $Q\in \Theta'_{p}$.
\end{proof}

\begin{lemma}\label{L11-11}  Let $(\mathcal{A},\Theta)$ be a length category.  Take a morphism $f:M\rightarrow N$ in $\mathcal{A}$ with $M\in\Theta_{m}$ and $N\in\Theta_{n}$.

$(1)$ If $m=n=1$, then either $f=0$ or $f$ is an isomorphism.

$(2)$  If $m=1$ and $n\geq2$, then $f=0$.

$(3)$ If $n=1$ and $m\geq2$, then $f=0$.

$(4)$ If $n\geq2$ and $m\geq2$, then either $f=0$ or $f$ is an isomorphism.
\end{lemma}
\begin{proof} (1) It follows from Lemma \ref{L-3-77}(1).

$(2)$ Since $M\in\Theta_{1}$ and $N\in\Theta'_{n}$, we get $f=0$.

$(3)$  Since $N\in\Theta_{1}$ and $M\in\Theta'_{m}$, we get $f=0$.

$(4)$ Take a $\Theta$-decomposition $(i_{f},X_{f},j_{f})$ of $f$.  Suppose that $f\neq0$. Then Lemma \ref{L-3-77} implies that  $m=\Theta(X_{f})=n$.  By using Lemma \ref{L2-4}, we obtain that $f$ is an isomorphism.
\end{proof}

As a consequence, we have the following characterization of $\Theta_{\infty}$.
\begin{proposition}\label{main9} Let $(\mathcal{A},\Theta)$ be a length category.
 Then $\Theta_{\infty}$ is a simple-minded system in $\mathcal{A}$.
\end{proposition}

\begin{proof} Firstly, we observe that $\Theta_{\infty}$ is a semibrick by Lemma \ref{L11-11}.  Suppose that $M\in\mathcal{A}$ with $\Theta(M)=m$. We claim that $M\in\Filt_{\mathcal{A}}(\Theta_{\infty})$. We proceed the proof by induction on $m$. The cases of $m=0,1$ are trivial. Now consider the general case. If $M\in\Theta'_{m}$, then $M\in\Theta_{\infty}$. If $M\notin\Theta'_{m}$, we have two cases:

$(\mathbf{Case~1})$. There exists a non-zero morphism $a:X\rightarrow M$ such that $X\in\Theta_{m-1}$. We take a $\Theta$-decomposition $(i_{a},X_{a},j_{a})$ for $a$ with $\Theta(X_{a})=t$. Clearly $t\leq m-1<m$ and $1\leq \Theta(\cone(j_{a}))< m$. By induction, we have $X_{a},\cone(j_{a})\in\Filt_{\mathcal{A}}(\Theta_{\infty})$. Since $\Filt_{\mathcal{A}}(\Theta_{\infty})$ is closed under extensions, we obtain $M\in\Filt_{\mathcal{A}}(\Theta_{\infty})$.

$(\mathbf{Case~2})$.  There exists a non-zero morphism  $a:M\rightarrow X$ such that $X\in\Theta_{m-1}$.  In analogy with Case 1, we have  $M\in\Filt_{\mathcal{A}}(\Theta_{\infty})$.

This shows that $\mathcal{A}\subseteq\Filt_{\mathcal{A}}(\Theta_{\infty})$.  The inclusion $\Filt_{\mathcal{A}}(\Theta_{\infty})\subseteq\mathcal{A}$ is obvious.
\end{proof}

When we say that an extriangulated category $\mathcal{A}$ is a (stable) length category, we always mean that there exists a (stable) length function $\Theta$ such that  $(\mathcal{A},\Theta)$ forms an length category. With this clarified, we can now state the main result of this section.

\begin{theorem}\label{main-L} Let $\mathcal{A}$ be an extriangulated category.  The following statements are equivalent:

$(1)$ The category $\mathcal{A}$ is a length category.

$(2)$ The category $\mathcal{A}$ has a simple-minded system.
\end{theorem}
\begin{proof} This follows from Theorem \ref{main1} and Proposition \ref{main9}.
\end{proof}

\subsection{Extriangulated length subcategories and semibricks}
Let $\mathcal{C}\subseteq \mathcal{A}$ be an extension-closed subcategory.  We define $\mathbb{E}_{\mathcal{C}}$ by the restriction of $\mathbb{E}$ onto $\mathcal{C}^{\rm op}\times \mathcal{C}$ and define $\mathfrak{s}_{\mathcal{C}}$ by restricting $\mathfrak{s}$. Then $(\mathcal{C},\mathbb{E}_{\mathcal{C}},\mathfrak{s}_{\mathcal{C}})$ is an extriangulated category (cf. \cite[Remark 2.18]{Na}).

\begin{definition}\label{Def25} Let $\mathcal{C}\subseteq \mathcal{A}$ be an extension-closed subcategory.

$(1)$ We say that $(\mathcal{C},\Theta)$ is a {\em length subcategory in $\mathcal{A}$} if  $((\mathcal{C},\mathbb{E}_{\mathcal{C}},\mathfrak{s}_{\mathcal{C}}),\Theta)$ is a length category.

$(2)$ We say that a length subcategory $(\mathcal{C},\Theta)$ in $\mathcal{A}$ is {\em length wide} if $\Theta_{1}=\Theta_{\infty}$.
\end{definition}

We say that a subcategory $\mathcal{C}$ is length (wide) if  there exists a length function $\Theta$ such that  $(\mathcal{C},\Theta)$ is a  length (wide) subcategory.

\begin{remark}\label{R-1} (1) Let $\mathcal{X}$ be a semibrick in $\mathcal{A}$. Using  Theorem \ref{main1} together with Remark \ref{r-1-3}, we can easily check that $(\Filt_{\mathcal{A}}(\mathcal{X}),l_{\mathcal{X}})$ is a length wide subcategory in $\mathcal{A}$.  Now, let $(\mathcal{C},\Theta)$ be a length subcategory in $\mathcal{A}$. By  Proposition \ref{main9}, $(\mathcal{C},l_{\Theta_{\infty}})$ is a  length wide subcategory. By this fact, a subcategory $\mathcal{C}$ is length if and only if $\mathcal{C}$ is length wide.

 $(2)$   Let $\mathcal{C}$ be a length subcategory in an exact category. Combining \cite[Exercise 8.6]{Bu} with Theorem \ref{main9}, we obtain that $\mathcal{C}$ is an abelian length category with the natural exact structure. Then by (1), it is easy to see that length subcategories in an exact category are just length wide subcategories in the sense of \cite[Definition 2.4]{En}.
 \end{remark}

Let $(\mathcal{C},\Theta)$ be a length subcategory of $\mathcal{A}$. A non-zero object $M$ in $\mathcal{C}$ is called a {\em simple} object if there does not exist a $\Theta$-stable $\mathbb{E}$-triangle $A\stackrel{}\rightarrowtail M\stackrel{}\twoheadrightarrow B\stackrel{}\dashrightarrow$ in $\mathcal{C}$ such that $A,B\neq0$. Let $\cim(\mathcal{C})$ be the set of isomorphism classes of simple objects in $(\mathcal{C},\Theta)$.  We have the following basic observation.

\begin{lemma}\label{L-3-17} Let $(\mathcal{C},\Theta)$ be a length subcategory in $\mathcal{A}$.

$(1)$ $\Theta_{1}=\Theta_{\infty}$ if and only if $\Theta=l_{\Theta_{1}}$.

$(2)$ If $(\mathcal{C},\Theta)$ is length wide, then $\Theta_{1}=\cim(\mathcal{C})$.

$(3)$ $(\mathcal{C},\Theta)$ is length wide if and only if $\Theta=l_{{\rm sim(\mathcal{C})}}$.
\end{lemma}
\begin{proof}
(1) Suppose that $\Theta_{1}=\Theta_{\infty}$. Then $\Theta_{1}$ is a simple-minded system in $\mathcal{A}$ by Theorem \ref{main9}. Take an object $M\in \mathcal{A}$ with $\Theta(M)=m$. The case of $m=0$ is trivial.  If $m=1$, then $M\in\Theta_{1}$ and thus $l_{\Theta_{1}}(M)=1$. Now consider the general case. We note that $\Theta_{1}^{\perp}\bigcap{^{\perp}}\Theta_{1}=\emptyset$. Without loss of generality, we may assume that there exists a non-zero morphism $f:X\rightarrow M$ such that $\Theta(X)=1$. By Lemma \ref{L-3-77}, $f$ is a $\Theta$-inflation and thus $\Theta(\cone(f))=m-1$. On the other hand, by Proposition \ref{p-1-2}, we have $l_{\Theta_{1}}(M)=1+l_{\Theta_{1}}(\cone(f))$. By induction hypothesis, we have $\Theta(\cone(f))=l_{\Theta_{1}}(\cone(f))$ and conclude that $\Theta(M)=l_{\Theta_{1}}(M)$. Conversely, if $\Theta=l_{\Theta_{1}}$, then $\Theta_{1}^{\perp}\bigcap{^{\perp}}\Theta_{1}=\emptyset$ by Remark \ref{r-1-3}. This implies that $\Theta_{1}=\Theta_{\infty}$.

$(2)$ Obviously, $\Theta_{1}\subseteq\cim(\mathcal{C})$. Take $L\in\cim(\mathcal{C})$. By Remark  \ref{r-1-3}, there exists an  $\mathbb{E}$-triangle $\eta:X\stackrel{}{\longrightarrow}L\stackrel{}{\longrightarrow}M\stackrel{}\dashrightarrow$ with $l_{\Theta_{1}}(X)=1$ and $l_{\Theta_{1}}(\cone(f))=l_{\Theta_{1}}(M)-1$. By (1), we have $\Theta(X)+\Theta(L)=\Theta(M)$ and thus $\eta$ is $\Theta$-stable. Since $L\in\cim(\mathcal{C})$, we get $X\cong L$ by Lemma \ref{L2-4}. This implies that $\cim(\mathcal{C})\subseteq\Theta_{1}$.

$(3)$ If $(\mathcal{C},\Theta)$ is length wide, then $\Theta_{1}=\Theta_{\infty}$ by definition. By (1) and (2), we have $\Theta=l_{\Theta_{1}}=l_{{\rm sim(\mathcal{C})}}$. Conversely, if  $\Theta=l_{{\rm sim(\mathcal{C})}}$, then  $\cim(\mathcal{C})$ is a simple-minded system in $\mathcal{C}$ by Proposition \ref{main9}. Then $(\mathcal{C},l_{{\rm sim(\mathcal{C})}})$ is length wide by Remark \ref{R-1}.
\end{proof}

We are now ready to present the main result of this subsection, which establishes a bijection between semibricks and length wide subcategories in extriangulated categories, revealing their intrinsic connection.

\begin{theorem}\label{main-s} Let $\mathcal{A}$ be an extriangulated category.
Then there exists a bijection between the following two sets.

$(1)$ The set of  semibricks $\mathcal{X}$ in $\mathcal{A}$.

$(2)$ The set of  length wide subcategories $(\mathcal{C},\Theta)$ in $\mathcal{A}$.

The mutually inverse maps are given by $\mathcal{X}\mapsto (\Filt_{\mathcal{A}}(\mathcal{X}),l_{\mathcal{X}})$ and $(\mathcal{C},\Theta)\mapsto \cim(\mathcal{C})$.
\end{theorem}

\begin{proof} For a semibrick $\mathcal{X}$ in $\mathcal{A}$, $(\Filt_{\mathcal{A}}(\mathcal{X}),l_{\mathcal{X}})$ is a length wide subcategory by Remark \ref{R-1}. In this case, we have $\mathcal{X}=\cim(\Filt_{\mathcal{A}}(\mathcal{X}))$ by Lemma \ref{L-3-17}. Conversely, for a length wide subcategory  $(\mathcal{C},\Theta)$, we have $\Theta_{\infty}=\Theta_{1}=\cim(\mathcal{C})$ and $\Theta=l_{{\rm sim(\mathcal{C})}}$. Then  $\cim(\mathcal{C})$ is a simple-minded system in $\mathcal{C}$  by Proposition \ref{main9}. This implies that $(\mathcal{C},\Theta)=(\Filt_{\mathcal{A}}(\cim(\mathcal{C})),l_{{\rm sim(\mathcal{C})}})$.
\end{proof}

\begin{remark} For a finite dimensional algebra $\Lambda$ over a field,
the classic result by Ringel \cite{Ri} tells us that there exists
a bijection between semibricks and wide subcategories in  $\mod \Lambda$.
This bijection was generalized to exact categories in \cite[Theorem 2.5]{En},
 which explicitly establishes a bijection between semibricks and length wide subcategories
 in an exact category. By using Remark \ref{R-1}, we can recover \cite[Theorem 2.5]{En}.
\end{remark}

\subsection{Stable extriangulated length categories} As before, $(\mathcal{A}, \mathbb{E},\mathfrak{s})$ is an extriangulated category.

Let $(\mathcal{A},\Theta)$ be a length category. The length function $\Theta$ is not stable in general. To solve this, we consider the induced extriangulated structures by using length functions. Firstly, we recall from \cite{HLN} that a functor $\mathbb{F}$
is called a {\em subfunctor} of $\mathbb{E}$  if it satisfies the following conditions:

$\bullet$ For any $M,N\in \mathcal{A}$, $\mathbb{F}(M,N)$ is a subset of  $\mathbb{E}(M,N)$.

$\bullet$ For any $f\in\Hom_{\mathcal{A}}(M',M)$, $g\in\Hom_{\mathcal{A}}(N,N')$, we have $\mathbb{F}(f,g)=\mathbb{E}(f,g)|_{\mathbb{F}(M,N)}$.

A subfunctor is said to be an {\em additive subfunctor} if $\mathbb{F}(M,N)\subseteq\mathbb{E}(M,N)$ is an abelian subgroup for any $M,N\in \mathcal{A}$. For any $A,B\in \mathcal{A}$, we define $\mathbb{E}_{\Theta}(A,B)$ to be the subset of $\mathbb{E}(A,B)$  consisting of all $\Theta$-extensions. In particular, if $\Theta$ is stable, then $\mathbb{E}_{\Theta}=\mathbb{E}$. We define $\mathfrak{s}_{\Theta}$ as the restriction of $\mathfrak{s}$ on $\mathbb{E}_{\Theta}$.
Let us begin with the following observation.

\begin{lemma}\label{4-77} Let $f:X\rightarrow Y$ and  $g:Y\rightarrow Z$ be any composable pair of morphisms in $\mathcal{A}$. If $gf$ is a $\Theta$-inflation, then so is $f$. Dually, if $gf$ is a $\Theta$-deflation, then so is $g$.
\end{lemma}
\begin{proof} We  only prove the first statement and the second statement can be proved dually. Firstly, we take a $\Theta$-decomposition $(i_{f},X_{f},j_{f})$ of $f$. Applying Lemma \ref{L-1}, we have the following commutative diagram:
\begin{equation*}
\xymatrix{
\cocone(i_{f})~\ar@{=}[d] \ar@{>->}[r] & X\vphantom{\big|}  \ar@{>->}[d]^{gf} \ar@{->>}[r]^{i_{f}} & X_{f} \vphantom{\bigg|}\ar@{>->}[d]^{} \ar@{-->}[r]^-{}&\\
\cocone(i_{f})~\ar@{>->}[r] & Z\ar@{->>}[d]^{} \ar@{->>}[r]^{s} & K \ar@{->>}[d]\ar@{-->}[r]^-{}& \\
  &   \cone(gf)  \ar@{-->}[d]^-{}\ar@{=}[r] & \cone(gf). \ar@{-->}[d]^-{}\\
   &  &  }
\end{equation*}
Since $gj_{f}i_{f}=gf$, there exists a morphism $l:K\rightarrow Z$ such that $ls=1_{Z}$. Then $s$ is a section and then $\Theta(Z)=\Theta(K)$. This implies that $i_{f}$ is an isomorphism and thus $f\cong j_{f}$ is a $\Theta$-inflation.
\end{proof}

\begin{lemma}\label{L-5-1} Let $\delta\in\mathbb{E}_{\Theta}(A,B)$. For any $a:A'\rightarrow A$  and $b:B\rightarrow B'$, we have $a^{\ast}\delta\in\mathbb{E}_{\Theta}(A',B)$ and $b_{\ast}\delta\in\mathbb{E}_{\Theta}(A,B')$. In particular, $\mathbb{E}_{\Theta}$ is an additive subfunctor of $\mathbb{E}$.
\end{lemma}
\begin{proof} Consider the following commutative diagram
\begin{equation*}
\xymatrix{
 B\ar[d]^{b} \ar@{>->}[r]^-{} & K\ar[d]_-{} \ar@{->>}[r]^-{f} & A \ar@{=}[d]_{} \ar@{-->}[r]^-{\delta}  &  \\
 B'\ar[r]^-{} & K' \ar[r]^-{g} & A \ar@{-->}[r]^-{b_{\ast}\delta}&.  }
\end{equation*}
Since $f$ is a $\Theta$-deflation, so is $g$ by Lemma \ref{4-77}. By using Lemma \ref{L2-4}, we have $b_{\ast}\delta\in\mathbb{E}_{\Theta}(A,B')$. Similarly, we deduce that $a^{\ast}\delta\in\mathbb{E}_{\Theta}(A',B)$.
\end{proof}

 We define $\mathbb{E}_{\mathcal{C}}$ as the restriction of $\mathbb{E}$ onto $\mathcal{C}^{\rm op}\times \mathcal{C}$ and define $\mathfrak{s}_{\mathcal{C}}$ by restricting  $\mathfrak{s}$. Our aim in this subsection is to prove the following result, which will be used in Section 4.

\begin{proposition}\label{11-11} Let $((\mathcal{A}, \mathbb{E},\mathfrak{s}),\Theta)$ be
a length category. Then  $((\mathcal{A}, \mathbb{E}_{\Theta},\mathfrak{s}_{\Theta}),\Theta)$ is a stable length category.
\end{proposition}

\begin{proof} First of all, $\mathbb{E}_{\Theta}$ is an additive subfunctor of $\mathbb{E}$ by Lemma \ref{L-5-1}. Using  \cite[Proposition 3.14]{HLN} together with Lemma \ref{L-1}, we obtain that the triple $(\mathcal{A}, \mathbb{E}_{\Theta},\mathfrak{s}_{\Theta})$ is an extriangulated category. Since $((\mathcal{A}, \mathbb{E},\mathfrak{s}),\Theta)$ is a length category, we can easily check that $((\mathcal{A}, \mathbb{E}_{\Theta},\mathfrak{s}_{\Theta}),\Theta)$  is a stable length category.
\end{proof}

\subsection{Examples of extriangulated length categories}
In this subsection, we present examples of length categories that involve several topics discussed in the preceding subsections.

\begin{example}Let $\Lambda$ be a  finite dimensional algebra.  As we discussed in Example \ref{E-3-10}, $\mod \Lambda$ is a stable length category. The length subcategories in $\mod \Lambda$  are just length wide subcategories by Remark \ref{R-1}. In particular, these subcategories are stable.
\end{example}

\begin{example}
Let $\mathcal{T}$ be a Hom-finite Krull-Schmidt triangulated category with shift functor $[1]$.  We say that a semibrick $\mathcal{S}$ in $\mathcal{T}$ is a {\em simple-minded collection} if $\mathcal{S}$ is a finite set, $\mathcal{S}$ generates $\mathcal{T}$ as triangulated category and $\Hom(S,S[n])=0$ for $n<0$. An {\em algebraic t-structure} on $\mathcal{T}$ is a $t$-structure $(\mathcal{T}^{\leq0},\mathcal{T}^{\geq0})$  such that the {\em heart} $\mathcal{H}:=\mathcal{T}^{\leq0}\cap\mathcal{T}^{\geq0}$ is an abelian length category with finitely many isomorphism classes of simple objects.  There is a bijection between isomorphism classes of simple-minded collections and algebraic $t$-structures. More precisely, for a simple-minded collection $\mathcal{S}$ in $\mathcal{T}$, the corresponding algebraic $t$-structure is $(\Filt_{\mathcal{T}}(S)[\geq0],\Filt_{\mathcal{T}}(S)[\leq0])$, where
$$\Filt_{\mathcal{T}}(\mathcal{S})[\geq0]:=\Filt_{\mathcal{T}}(\bigcup_{S\in\mathcal{S},i\geq0} S[i])~\text{and}~\Filt_{\mathcal{T}}(\mathcal{S})[\leq 0]:=\Filt_{\mathcal{T}}(\bigcup_{S\in\mathcal{S},i\leq0} S[i]).$$
Thus  $\Filt_{\mathcal{T}}(\mathcal{S}):=\Filt_{\mathcal{T}}(\mathcal{S})[\geq0]\cap \Filt_{\mathcal{T}}(\mathcal{S})[\leq0]$ is a stable length subcategory in  $\mathcal{T}$. As a consequence, a simple-minded collection is  proper.
\end{example}

\begin{example}\label{Derived}
Let $\Lambda$ be a finite dimensional algebra of finite global dimension. Recall that the {\em trivial extension} of $\Lambda$ is a Frobenius algebra $\mathrm{T}(\Lambda)=\Lambda\oplus D\Lambda$, where $D$ is the standard linear duality.  It is well-know that  there is an equivalence of triangulated categories $D^{b}(\Lambda)\cong \underline{\mathbf{gr}}$-$\mathrm{T}(\Lambda)$, where $\underline{\mathbf{gr}}$-$\mathrm{T}(\Lambda)$ is the stable category of $\mathbb{Z}$-graded $\mathrm{T}(\Lambda)$-modules. Let $\{S_{i}~|~1\leq i\leq n\}$ be the set of isomorphism classes of simple $\mathrm{T}(\Lambda)$-modules. Then the set $\mathcal{S}=\{S_{ij}~|~1\leq i\leq n,j\in \mathbb{Z}\}$ is a  simple-minded system in  $\underline{\mathbf{gr}}$-$\mathrm{T}(\Lambda)$, where $S_{ij}$ denotes the simple module $S_{i}$ concentrated in degree 0 with grading shift $j$. By \cite[Example 2]{Du}, the $\mathrm{T}(\Lambda)$-module $S_{ij}$ corresponds to $v^{j}(S_{i}[j])$ in $D^{b}(\Lambda)$, where $v$ denotes the Serre functor of $D^{b}(\Lambda)$. Then Theorem \ref{main1} implies that $(D^{b}(\Lambda),l_{\mathcal{S}})$ is a length category. Now we explain the notions by a concrete example.

Let $\Lambda$ be the path algebra of the quiver $1\longrightarrow2\longrightarrow3$. The Auslander-Reiten quiver $\Gamma$ of the bounded derived category $D^{b}(\Lambda)$ is as follows:
\begin{equation*}
\xymatrix@!=0.5pc{
   && S_3[-1]\ar[dr]  && S_2[-1]\ar[dr] && S_1[-1]\ar[dr] && P_1\ar[dr] && \\
   &&\cdots\cdots\quad& P_2[-1]\ar[dr]\ar[ur] && I_2[-1]\ar[ur]\ar[dr] && P_2\ar[ur]\ar[dr] && I_2\ar[dr] & \cdots\cdots\\
   &&&& P_1[-1]\ar[ur] && S_3\ar[ur] && S_2\ar[ur] && S_1}
\end{equation*}
We identify $S_{ij}\in\underline{\mathbf{gr}}$-$\mathrm{T}(\Lambda)$ with $v^{j}(S_{i}[j])$ in $D^{b}(\Lambda)$. We calculate:
$$
\begin{tabular}{|p{0.5cm}|p{1.2cm}|p{2,0cm}|c|p{2.0cm}|}
\hline
 & $j=4k$&    $j=4k+1$    &     $j=4k+2$ & $j=4k+3$\\
\hline
$S_{3j}$& $S_{3}[6k]$    &   $P_{1}[6k+1]$   &   $S_{1}[6k+2]$    &     $S_{2}[6k+4]$      \\
\hline
$S_{2j}$&     $S_{2}[6k]$           &   $S_{3}[6k+2]$    &    $P_{1}[6k+3]$     &     $S_{1}[6k+4]$          \\
\hline
$S_{1j}$&    $S_{1}[6k]$             &    $S_{2}[6k+2]$  &    $S_{3}[6k+4]$      &          $P_{1}[6k+5]$             \\
\hline
\end{tabular}
$$

Therefore, the set $\mathcal{S}=\{S_{ij}~|~1\leq i\leq n,j\in \mathbb{Z}\}$ is consisting of the isomorphism classes of objects in the bottom row of $\Gamma$. This implies that $(D^{b}(\Lambda),l_{\mathcal{S}})$ is a length category. The length structure  possesses the following properties:

$\bullet$  $l_{\mathcal{S}}$ is not stable.

$\bullet$  If $M\notin \mathcal{S}$, then $l_{\mathcal{S}}(M[i])=2$ for any $i\in\mathbb{Z}$.

$\bullet$ If  $M\in \mathcal{S}$, then $l_{\mathcal{S}}(M[i])=3$ if $i$ is odd; $l_{\mathcal{S}}(M[i])=1$ if $i$ is even.
\end{example}

The following example demonstrates that a derived category can exhibit significantly different length structures.

\begin{example}\label{E-3-26} We keep the notation of Example \ref{Derived}. Let $\mathcal{X}$ be the set consisting of the isomorphism classes of objects in the top row of $\Gamma$, i.e. $$\mathcal{X}=\bigcup_{i=2k,k\in \mathbb{Z}}\{P_{1}[i-1],S_{3}[i],S_{2}[i],S_{1}[i]\}.$$
It is clear that $\mathcal{X}$ is a simple-minded system in $D^{b}(\Lambda)$. Therefore, we have two distinct length structures $(D^{b}(\Lambda),l_{\mathcal{X}})$ and $(D^{b}(\Lambda),l_{\mathcal{S}})$.
\end{example}

We also provide an example of a length category, which is neither exact nor triangulated.

\begin{example}\label{E-2-22} We retain the notation used in Example \ref{Derived} and set $$\mathcal{Y}=\{S_{2}[-1],S_{1}[-1],P_{1}\}.$$ Then  the Auslander-Reiten quiver of $\mathcal{A}:=\Filt_{D^{b}(\Lambda)}(\mathcal{Y})$ is given by
\begin{equation*}
\xymatrix@!=0.5pc{
    S_2[-1]\ar[dr] && S_1[-1]\ar[dr] && P_1  \\
   & I_2[-1]\ar[ur]\ar[dr] && P_2\ar[ur] & \\
     && S_3\ar[ur] &&}
\end{equation*}
Since $P_{2}\rightarrow P_{1}$ is a monomorphism in $\mod(\Lambda)$, we easily see that $(\mathcal{A},l_{\mathcal{Y}})$ is a stable length  subcategory  in $D^{b}(\Lambda)$ which is neither exact nor triangulated.
\end{example}

\section{Torsion classes in extriangulated length categories}
Our aim in this section is to establish  a lattice theoretical framework of torsion classes in length categories, which is  a generalization of some classical results over a finite dimensional algebra (c.f. \cite{De}). In this section, we always assume that  $((\mathcal{A}, \mathbb{E},\mathfrak{s}),\Theta)$ is a  length category.

\subsection{Torsion classes}
Given subcategories $\mathcal{S},\mathcal{T}\subseteq \mathcal{A}$, we define
 $$\Fac_{\Theta}(\mathcal{S})=\{M\in \mathcal{A}~|~\text{there exists a $\Theta$-deflation}~S\twoheadrightarrow M~\text{for some}~S\in \mathcal{S}\},$$
$$\Sub_{\Theta}(\mathcal{S})=\{M\in \mathcal{A}~|~\text{there exists a $\Theta$-inflation}~M\rightarrowtail S~\text{for some}~S\in \mathcal{S}\},$$
 $$\mathcal{S}\ast_{\Theta} \mathcal{T}=\{M\in \mathcal{A}~|~\text{there exists a stable $\mathbb{E}$-triangle}~T\stackrel{}\rightarrowtail M\stackrel{}\twoheadrightarrow F\stackrel{}\dashrightarrow~$$
$$\text{with}~T\in\mathcal{S},F\in\mathcal{T} \}.$$

We introduce the following notions, which play a central role in this section.
\begin{definition}\label{Torsion} A {\em torsion class}  in $(\mathcal{A},\Theta)$ is a subcategory $\mathcal{T}$ of $\mathcal{A}$ such that $\Fac_{\Theta}(\mathcal{T})\subseteq\mathcal{T}$ and $\mathcal{T}\ast_{\Theta}\mathcal{T}\subseteq\mathcal{T}$. The set of torsion classes in $(\mathcal{A},\Theta)$ is denoted by $\tors_{\Theta}(\mathcal{A})$. Dually, a subcategory $\mathcal{F}$ of $\mathcal{A}$ is called a {\em torsion-free class} in $(\mathcal{A},\Theta)$ if $\Sub_{\Theta}(\mathcal{T})\subseteq\mathcal{T}$ and $\mathcal{T}\ast_{\Theta}\mathcal{T}\subseteq\mathcal{T}$. The set of torsion-free classes in $(\mathcal{A},\Theta)$ is denoted by $\torf_{\Theta}(\mathcal{A})$.
\end{definition}

If $(\mathcal{A},\Theta)=(\mod \Lambda,l_{{\rm sim({\rm mod} \Lambda)}})$ for a finite dimensional algebra $\Lambda$, then Definition \ref{Torsion} coincides with the usual one.

\begin{definition} A  {\em torsion pair}  $(\mathcal{T},\mathcal{F})$ in $(\mathcal{A},\Theta)$ is a pair of subcategories $\mathcal{T},\mathcal{F}\in \mathcal{A}$ such that $\Hom_{\mathcal{A}}(\mathcal{T},\mathcal{F})=0$, and for any $M\in\mathcal{A}$, there exists a $\Theta$-stable $\mathbb{E}$-triangle $T\stackrel{}\rightarrowtail M\stackrel{}\twoheadrightarrow F\dashrightarrow$ with $T\in\mathcal{T}$ and $F\in\mathcal{F}$.
\end{definition}

When $\Theta$ is stable, the torsion pairs in $(\mathcal{A},\Theta)$ coincide with the usual torsion pairs.  The following provides a useful characterization of torsion pairs in a length category.

\begin{lemma}\label{L-3-3} Let  $(\mathcal{T},\mathcal{F})$ be a torsion pair in $(\mathcal{A},\Theta)$. The following statements hold:

$(1)$ $\mathcal{T}^{\perp}=\mathcal{F}$ and $\mathcal{T}={^{\perp}}\mathcal{F}$.

$(2)$  $\mathcal{T}$ and $\mathcal{F}$ are extension-closed.

$(3)$ $\mathcal{T}$ is a torsion class in $(\mathcal{A},\Theta)$.

$(3')$ $\mathcal{F}$ is a torsion-free class in $(\mathcal{A},\Theta)$.
\end{lemma}

\begin{proof} (1) The inclusion $\mathcal{F}\subseteq\mathcal{T}^{\perp}$ is obvious. For each $M\in\mathcal{T}^{\perp}$, there exists a $\Theta$-stable $\mathbb{E}$-triangle $T\stackrel{f}\rightarrowtail M\stackrel{}\twoheadrightarrow F\dashrightarrow$ with $T\in\mathcal{T}$ and $F\in\mathcal{F}$. Since $f=0$, we obtain $M\cong F\in \mathcal{F}$. This proves $\mathcal{T}^{\perp}=\mathcal{F}$. The proof of $\mathcal{T}={^{\perp}}\mathcal{F}$ is similar.

$(2)$ We take an $\mathbb{E}$-triangle $T_{1}\stackrel{}\longrightarrow M\stackrel{}\longrightarrow T_{2}\dashrightarrow$ with each $T_{i}\in\mathcal{T}$. Applying $\Hom_{\mathcal{A}}(-,\mathcal{F})$ to it, we deduce $\Hom_{\mathcal{A}}(M,\mathcal{F})=0$. By (1), we get $M\in\mathcal{T}$ and thus $\mathcal{T}\ast\mathcal{T}\subseteq\mathcal{T}$. Similarly, the second
statement can be proved.

$(3)$ For any $M\in\Fac_{\Theta}(\mathcal{T})$, there exists an $\mathbb{E}$-triangle $T'\stackrel{}\rightarrowtail T\stackrel{a}\twoheadrightarrow M\dashrightarrow$ with  $T\in\mathcal{T}$. Since  $(\mathcal{T},\mathcal{F})$ is a torsion pair, we have an $\mathbb{E}$-triangle $T''\stackrel{}\rightarrowtail M\stackrel{b}\twoheadrightarrow F\dashrightarrow$ with  $T''\in\mathcal{T}$ and $F\in \mathcal{F}$. Since $ba=0$ is a $\Theta$-deflation, we have $F\cong0$ and thus $M\cong T''\in \mathcal{T}$. Similarly, we can prove $(3)'$.
\end{proof}

\begin{lemma}\label{L-2-1}  $(1)$ If $\mathcal{T}\in\tors_{\Theta}(\mathcal{A})$, then $(\mathcal{T},\mathcal{T}^{\perp})$ is a torsion pair in $(\mathcal{A},\Theta)$.

$(2)$  If $\mathcal{F}\in\torf_{\Theta}(\mathcal{A})$, then $(^{\perp}\mathcal{F},\mathcal{F})$ is a torsion pair in $(\mathcal{A},\Theta)$.

\end{lemma}
\begin{proof} We only prove (1), since (2) can be proved in a similar way.

Take any object $M\in\mathcal{A}$. Without loss of generality, we assume that $M\notin \mathcal{T}^{\perp}$. Then there exists a non-zero morphism $f:N\rightarrow M$ for some $N\in \mathcal{T}$. If $\Theta(M)=1$, by Lemma \ref{L-3-77},  $f$ is a $\Theta$-deflation and thus $M\in\Fac_{\Theta}(\mathcal{T})$. For the general case, we take a $\Theta$-decomposition $(i_{f},X_{f},j_{f})$ for $f$. Note that $X_{f}\in \mathcal{T}$ since $\mathcal{T}$ is a torsion class. If $\Theta(\cone(j_{f}))=\Theta(M)$, then $X_{f}\cong0$ and thus $f=0$. This is a contradiction. Therefore, by induction hypothesis, there exists an $\mathbb{E}$-triangle $T\rightarrowtail \cone(j_{f})\twoheadrightarrow F\stackrel{}\dashrightarrow$ with $T\in \mathcal{T}$ and $F\in \mathcal{T}^{\perp}$. Applying $\rm (ET4)^{op}$ yields the following  commutative diagram
\begin{equation}\label{L-3}
\begin{array}{l}
\xymatrix{
 X_{f}~\ar@{=}[d] \ar@{>->}[r] & K \vphantom{\big|} \ar@{>->}[d]^{} \ar@{->>}[r]^{} & T\vphantom{\big|}  \ar@{>->}[d] \ar@{-->}[r]^-{}&\\
  X_{f}~\ar@{>->}[r] & M\ar@{->>}[d]^{} \ar@{->>}[r]^{} & \cone(j_{f}) \ar@{->>}[d]\ar@{-->}[r]^-{}& \\
  &  F  \ar@{-->}[d]^-{}\ar@{=}[r] & F.\ar@{-->}[d]^-{}\\
   &  &  }
\end{array}
\end{equation}
Since $X_{f},T\in \mathcal{T}$, we have $K\in \mathcal{T}$. Then the second column in (\ref{L-3}) gives the desired $\mathbb{E}$-triangle.
\end{proof}

\begin{remark}\label{R-3-5} Combing Lemma \ref{L-3-3} with Lemma \ref{L-2-1}, we get that the torsion  (resp. torsion-free) classes can also be defined as follows: A  torsion (resp. torsion-free) class  in $(\mathcal{A},\Theta)$ is a subcategory $\mathcal{T}$ such that $\Fac_{\Theta}(\mathcal{T})\subseteq\mathcal{T}$ (resp.  $\Sub_{\Theta}(\mathcal{T})\subseteq\mathcal{T}$) and $\mathcal{T}\ast\mathcal{T}\subseteq\mathcal{T}$.
\end{remark}

There is a natural bijection between torsion classes and torsion-free classes in a length category,
described as follows.
\begin{proposition}\label{L-2-2} There exist mutually inverse bijections
$$
\xymatrix@C=3.5pc{ \tors_{\Theta}(\mathcal{A})\ar@<-1ex>[r]_-{(-)^{\perp}}&
\torf_{\Theta}(\mathcal{A})\ar@<-1ex>[l]_-{{^{\perp}}(-)}.        }
$$
\end{proposition}
\begin{proof}   Let $\mathcal{T}\in\tors(\mathcal{A})$. It is easy to check that $\mathcal{T}^{\perp}$ is closed under extensions. For $M\in\Sub_{\Theta}(\mathcal{T}^{\perp})$, there exists a $\Theta$-inflation $f:M\rightarrowtail N$ with $N\in\mathcal{T}^{\perp}$. By Lemma \ref{L-2-1}, there exists an $\mathbb{E}$-triangle  $T\stackrel{a}\rightarrowtail M\twoheadrightarrow F\stackrel{}\dashrightarrow$ with $T\in \mathcal{T}$ and $F\in\mathcal{T}^{\perp}$. Note that $fa=0$ is a $\Theta$-inflation, this implies $T=0$ and thus $M\cong F\in \mathcal{T}^{\perp}$. Hence, we have shown that $\mathcal{T}^{\perp}\in\torf_{\Theta}(\mathcal{A})$. Dually, if $\mathcal{F}\in\torf(\mathcal{A})$, then $^{\perp}\mathcal{F}\in\tors(\mathcal{A})$ by a similar argument as above. By Lemma \ref{L-3-3} and Lemma \ref{L-2-1}, it is easy to see that $(-)^{\perp}$ and ${^{\perp}(-)}$ are mutually inverse bijections.
\end{proof}

\subsection{Specifying a torsion class} Let $\mathcal{X}$ be a class of objects in $\mathcal{A}$. We denote by $\Filt_{\Theta}(\mathcal{X})$ the subcategory consisting of all objects $M$ admitting a finite filtration of the form
\begin{equation*}
0=M_{0}\stackrel{f_{0}}\rightarrowtail M_{1}\stackrel{f_{1}}\rightarrowtail M_{2}\rightarrowtail\cdots\stackrel{f_{n-1}}\rightarrowtail M_{n}=M
\end{equation*}
with $f_{i}$ being a $\Theta$-inflation and $\cone(f_{i})\in\mathcal{X}$ for any $0\leq i\leq n-1$. For each object $M\in\Filt_{\Theta}(\mathcal{X})$, the minimal length of such filtration of $M$ is denoted by $l_{\mathcal{X}}^{\Theta}(M)$. Clearly,  $\Filt_{\Theta}(\mathcal{X})$ is a subcategory of  $\mathbf{Filt_{\mathcal{A}}(\mathcal{X})}$.  If $\Theta$ is stable, then  $\Filt_{\Theta}(\mathcal{X})=\mathbf{Filt_{\mathcal{A}}(\mathcal{X})}$.

\begin{remark}\label{4-8}First of all, we note that   $((\mathcal{A}, \mathbb{E}_{\Theta},\mathfrak{s}_{\Theta}),\Theta)$  is a stable length category by Proposition \ref{11-11}. By construction, $\Filt_{\Theta}(\mathcal{X})$ is precisely the filtration subcategory generated by $\mathcal{X}$ in the extriangulated category $(\mathcal{A}, \mathbb{E}_{\Theta},\mathfrak{s}_{\Theta})$. Hence, we can apply Proposition \ref{p-1-1} to $\Filt_{\Theta}(\mathcal{X})$.
\end{remark}

For a subcategory $\mathcal{S}\subseteq \mathcal{A}$, we define
$$\mathrm{T}_{\Theta}(\mathcal{S}):=\bigcap_{\mathcal{T}\in{\rm tors_{\Theta}}(\mathcal{A});\mathcal{S}\subseteq\mathcal{ T}}\mathcal{T}~\text{and}~\mathrm{F}_{\Theta}(\mathcal{S}):=\bigcap_{\mathcal{F}\in{\rm torf_{\Theta}}(\mathcal{A});\mathcal{S}\subseteq\mathcal{F}}\mathcal{F}.$$
Clearly, $\mathrm{T}_{\Theta}(\mathcal{S})$ and $\mathrm{F}_{\Theta}(\mathcal{S})$ are the smallest torsion class and the smallest torsion-free class containing $\mathcal{S}$, respectively. It is well known that $\mathrm{T}=\Filt\circ\Fac$ and $\mathrm{F}=\Filt\circ\Sub$ in the setting of finitely generated module category. We can generalize this result as follows.
\begin{proposition}\label{P-2-6} Let $\mathcal{S}$ be a subcategory of $\mathcal{A}$. Then $$\mathrm{T}_{\Theta}(\mathcal{S})=\Filt_{\Theta}(\Fac_{\Theta}(\mathcal{S}))~{\text and}~\mathrm{F}_{\Theta}(\mathcal{S})=\Filt_{\Theta}(\Sub_{\Theta}(\mathcal{S})).$$
\end{proposition}
\begin{proof}

Firstly, we have $\Filt_{\Theta}(\Fac_{\Theta}(\mathcal{S}))\ast_{\Theta}\Filt_{\Theta}(\Fac_{\Theta}(\mathcal{S}))\subseteq\Filt_{\Theta}(\Fac_{\Theta}(\mathcal{S}))$  by Remark \ref{4-8}. It suffices to prove $M\in\Filt_{\Theta}(\Fac_{\Theta}(\mathcal{S}))$ for any $\Theta$-deflation $f:S\twoheadrightarrow M$ with $S\in \Filt_{\Theta}(\Fac_{\Theta}(\mathcal{S}))$. If $S\in\Fac_{\Theta}(\mathcal{S})$, then so is $M$. For the general case, by Remark \ref{4-8}, there exists an $\mathbb{E}$-triangle
 $$X_{1}\stackrel{g}\rightarrowtail S\stackrel{}\twoheadrightarrow S'\stackrel{}\dashrightarrow$$
with $X_{1}\in\Fac_{\Theta}(\mathcal{S})$ and $l_{\mathcal{X}}^{\Theta}(S')=l_{\mathcal{X}}^{\Theta}(S)-1$. Take a $\Theta$-decomposition $(i_{gf},X_{gf},j_{gf})$ of $gf$. We have the following commutative diagram
\begin{equation*}
\xymatrix{
  X_{1}~\ar@{->>}[d]^{i_{gf}} \ar@{>->}[r]^-{g} & S\ar@{->>}[d]_-{f} \ar@{->>}[r]^-{cf} & S' \ar[d]_{h} \ar@{-->}[r]^-{}  &  \\
  X_{gf}~ \ar@{>->}[r]^-{j_{gf}} & M \ar@{->>}[r]^-{c} & K \ar@{-->}[r]^-{}&.  }
\end{equation*}
Note that $h$ is a $\Theta$-deflation by Lemma \ref{4-77}. By  induction hypothesis, we get $X_{gf},K\in\Filt_{\Theta}(\Fac_{\Theta}(\mathcal{S}))$ and so is $M$.  We have shown that $\mathrm{T}_{\Theta}(\mathcal{S})=\Filt_{\Theta}(\Fac_{\Theta}(\mathcal{S}))$. Dually, we can prove $\mathrm{F}_{\Theta}(\mathcal{S})=\Filt_{\Theta}(\Sub_{\Theta}(\mathcal{S}))$.
\end{proof}

\subsection{Lattice of torsion classes}
First of all, we recall some notations from poset theory.

Let $P=(P,\leq)$ be a partially ordered set and let $x,y\in P$.

$\bullet$ We say an element $z\in P$ is a {\em join} of $x$ and $y$, denoted by $x\vee y$, if $z$ is the smallest element in the set $\{t\in P~|~t\geq x,t\geq y\}$. The {\em meet} $x\wedge y$ of $x$ and $y$ is the largest element in the set $\{t\in P~|~t\leq x,t\leq y\}$. We remark that the join or meet of two elements may not exist. For a subset $L$ of $P$, the join (resp. meet) of $L$ is the element $\bigvee L:=\bigvee_{L_{i}\in L} L_{i}$ (resp. $\bigwedge L:=\bigwedge_{L_{i}\in L} L_{i}$). Obviously, $\bigvee L$ (resp. $\bigwedge L$) is the unique minimal ({resp. maximal}) common upper (resp. lower) bound of elements in $L$.

$\bullet$ We say  $P$ is a {\em complete lattice} if every subset $S$ of $P$ admits  join and meet.  A morphism $\alpha:P\rightarrow L$ of complete lattices is a map such that $\alpha(\bigvee L)=\bigvee(\alpha(L))$ and $\alpha(\bigwedge L)=\bigwedge(\alpha(L))$ for any subset $L$ of $P$.

$\bullet$ Let $P$ be a complete lattice. We say an element $x\in P$ is {\em compact} if for any $L\subseteq P$ such that $x\leq\bigvee L$, there exists a finite subset $S\subseteq L$ such that $x\leq\bigvee_{S_{i}\in S}S_{i}$.

$\bullet$ The {\em interval} is a subset of $P$ which has the form $[x,y]:=\{z\in P~|~x\leq z\leq y\}$ for some $x,y\in P$. We say that $y$ is a {\em cover} of $x$ if $[x,y]=\{x,y\}$. The {\em Hasse quiver} Hasse$(P)$ has $P$ as the set of vertices and for any $x,y\in P$, there is an arrow $y\rightarrow x$ if and only if $y$ is a  cover of $x$.

Obviously, $\tors_{\Theta} (\mathcal{A})$ and $\torf_{\Theta} (\mathcal{A})$ are partially ordered sets with respect to the inclusion relation. The following basic observation on $\tors_{\Theta} (\mathcal{A})$ is useful.
\begin{lemma}\label{L-4-9} Let $\mathcal{U},\mathcal{T}$ be two torsion classes in $(\mathcal{A},\Theta)$. We have $\mathcal{U}\wedge\mathcal{T}=\mathcal{U}\cap\mathcal{T}$ and $\mathcal{U}\vee\mathcal{T}=\Filt_{\Theta}(\mathcal{U}\cup\mathcal{V})$.
\end{lemma}
\begin{proof} The first statement follows from the fact that $\tors_{\Theta} \mathcal{A}$ is closed under intersections. The second one follows from Proposition \ref{P-2-6}.
\end{proof}

\begin{definition}\label{semi} (\cite[Definition 2.8]{De}) We say a complete lattice $P$ is {\em completely semidistributive} if  for any $x\in P$ and $L\subseteq P$, the following hold:
\begin{itemize}
   \item [(1)] If $x\wedge y=x\wedge z$ for any $y,z\in L$, then $x\wedge(\bigvee L)=x\wedge y$ for any $y\in L$.
   \item [(2)] If $x\vee y=x\vee z$ for any $y,z\in L$, then $x\vee(\bigwedge L)=x\vee y$ for any $y\in L$.
  \end{itemize}
\end{definition}
\begin{definition}\label{semi} (\cite[Definition 2.9]{De}) We say that a complete lattice $P$ is {\em  algebraic} if for any $x\in P$, there exists a set $L$ of compact objects of $P$ such that $x=\bigvee L$.
\end{definition}

\begin{theorem}\label{main2} Let $(\mathcal{A},\Theta)$ be a  length category.

$(1)$ The set $\tors_{\Theta}(\mathcal{A})$ is a complete lattice with joins and meets given by
$$\bigvee \mathcal{L}={^{\perp}}(\bigcap_{\mathcal{T}_{i}\in \mathcal{L}} {\mathcal{T}_{i}}^{\perp})=\Filt_{\Theta}(\bigcup_{\mathcal{T}_{i}\in \mathcal{L}} \mathcal{T}_{i})~\text{and}~\bigwedge\mathcal{L}=\bigcap_{\mathcal{T}_{i}\in \mathcal{L}} {\mathcal{T}_{i}}$$
for any $\mathcal{L}\subseteq \tors_{\Theta}(\mathcal{A})$.

$(2)$ The set $\torf_{\Theta}(\mathcal{A})$ is a complete lattice with joins and meets given by
$$\bigvee \mathcal{L}'=(\bigcap_{\mathcal{F}_{i}\in \mathcal{L}'} {^{\perp}}{\mathcal{F}_{i}}){^{\perp}}=\Filt_{\Theta}(\bigcup_{\mathcal{F}_{i}\in \mathcal{L}'} \mathcal{F}_{i})~\text{and}~\bigwedge\mathcal{L}'=\bigcap_{\mathcal{F}_{i}\in \mathcal{L}'} {\mathcal{F}_{i}}$$
for any $\mathcal{L}'\subseteq \torf_{\Theta}(\mathcal{A})$.

$(3)$ We have an isomorphism of complete lattices $\tors_{\Theta} (\mathcal{A})\stackrel{}{\longrightarrow}\torf_{\Theta} (\mathcal{A})^{{\rm op}};~\mathcal{T}\mapsto \mathcal{T}^{\perp}$.

$(4)$  The complete lattice $\tors_{\Theta}(\mathcal{A})$ is  completely semidistributive.

$(5)$  The complete lattice $\tors_{\Theta}(\mathcal{A})$ is algebraic.
\end{theorem}
\begin{proof} (1)--(3) Since $\tors_{\Theta} \mathcal{A}$ and $\torf_{\Theta} \mathcal{A}$ are closed under arbitrary intersections, we clearly have
$$\bigwedge\mathcal{L}=\bigcap_{\mathcal{T}_{i}\in \mathcal{L}} {\mathcal{T}_{i}}~\text{and}~\bigwedge\mathcal{L}'=\bigcap_{\mathcal{F}_{i}\in \mathcal{L}'} {\mathcal{F}_{i}}.$$
Moreover, by Proposition \ref{L-2-2} and Lemma \ref{L-4-9}, one can easily check that
$$\bigvee \mathcal{L}={^{\perp}}(\bigcap_{\mathcal{T}_{i}\in \mathcal{L}} {\mathcal{T}_{i}}^{\perp})=\Filt_{\Theta}(\bigcup_{\mathcal{T}_{i}\in \mathcal{L}} \mathcal{T}_{i})~\text{and}~\bigvee \mathcal{L}'=(\bigcap_{\mathcal{F}_{i}\in \mathcal{L}'} {^{\perp}}{\mathcal{F}_{i}}){^{\perp}}=\Filt_{\Theta}(\bigcup_{\mathcal{F}_{i}\in \mathcal{L}'} \mathcal{F}_{i}).$$
Hence, $\tors_{\Theta}(\mathcal{A})$ and $\torf_{\Theta}(\mathcal{A})$ are complete lattices. By Proposition \ref{L-2-2} again, it is easy to see that $\tors (\mathcal{A})\stackrel{}{\rightarrow}\torf (\mathcal{A})^{{\rm op}};~\mathcal{T}\mapsto \mathcal{T}^{\perp}$ is an isomorphism of  complete lattices.

$(4)$ Let $\mathcal{U}\in\tors_{\Theta}(\mathcal{A})$ and $\mathcal{L}\subseteq\tors_{\Theta}(\mathcal{A})$ such that $\mathcal{U}\wedge \mathcal{T}=\mathcal{U}\wedge \mathcal{V}$ for any $\mathcal{T},\mathcal{V}\in \mathcal{L}$. The inclusion $\mathcal{U}\wedge \mathcal{T}\subseteq\mathcal{U}\wedge (\bigvee\mathcal{L})$ is obvious. For the other one, by Lemma \ref{L-4-9}, we have $\mathcal{U}\wedge (\bigvee\mathcal{L})=\mathcal{U}\cap\Filt_{\mathcal{A}}(\mathcal{X})$, where $\mathcal{X}=\bigcup_{\mathcal{L}_{i}\in \mathcal{L}} \mathcal{L}_{i}$. For any $0\neq L\in\mathcal{U}\wedge (\bigvee\mathcal{L})$, there exists a stable $\mathbb{E}$-triangle  $X\stackrel{}\rightarrowtail L\stackrel{}\twoheadrightarrow M\dashrightarrow$ with $X\in\mathcal{X}$ and $l_{\mathcal{X}}(M)=l_{\mathcal{X}}(L)-1$. If $l_{\mathcal{X}}(L)=1$, then $L\in \mathcal{U}\cap \mathcal{L}_{i}=\mathcal{U}\wedge \mathcal{T}$ for some $\mathcal{L}_{i}\in \mathcal{L}$. For the general case, we observe that $M\in\mathcal{U}\wedge (\bigvee\mathcal{L})$ since $\mathcal{U}$ is a torsion class. By induction hypothesis, we have $M\in\mathcal{U}\wedge \mathcal{T}$ and so is $L$. By duality, we can prove the second condition in Definition \ref{semi}.

$(5)$ For any $\mathcal{T}\in\tors_{\Theta}(\mathcal{A})$, we can check directly that $\mathcal{T}=\bigvee_{T\in\mathcal{T}} \mathrm{T}_{\Theta}(T)$. It suffices to prove $\mathrm{T}_{\Theta}(T)$ is  compact. Suppose that $\mathcal{L}\subseteq\tors_{\Theta}(\mathcal{A})$ and $\mathrm{T}_{\Theta}(T)\subseteq\mathcal{L}$. Since $T\in\bigvee\mathcal{L}=\Filt_{\Theta}(\bigcup_{\mathcal{L}_{i}\in \mathcal{L}}\mathcal{L}_{i})$, there exists a finite set $\mathcal{S}=\{\mathcal{L}_{1},\mathcal{L}_{2},\cdots,\mathcal{L}_{n}\}\subseteq \mathcal{L}$ such that $T\in \mathcal{L}_{1}\ast_{\Theta}\mathcal{L}_{2}\ast_{\Theta}\cdots \ast_{\Theta}\mathcal{L}_{n}$. Thus $\mathrm{T}_{\Theta}(T)\subseteq\mathrm{T}_{\Theta}(\bigcup_{i=1}^{n}L_{i})=\bigvee\mathcal{S}$. This shows that $\mathrm{T}_{\Theta}(T)$ is  compact.
 \end{proof}

\subsection{Brick labelling}
For any $\mathcal{U},\mathcal{T}\in\tors_{\Theta}(\mathcal{A})$, we denote by {\rm brick}$(\mathcal{U}^{\perp}\cap \mathcal{T})$ the set of isomorphism classes of bricks in $\mathcal{U}^{\perp}\cap \mathcal{T}$.
\begin{lemma}\label{L-2-3} Let $\mathcal{U}\subseteq\mathcal{T}$ be two torsion classes in $\mathcal{A}$. For any $0\neq M\in \mathcal{U}^{\perp}\cap \mathcal{T}$, there exists a non-zero endomorphism $f:M\rightarrow M$ admitting a $\Theta$-decomposition $(i_{f},X_{f},j_{f})$ with $X_{f}\in{\rm brick }(\mathcal{U}^{\perp}\cap \mathcal{T})$.
\end{lemma}
\begin{proof} If $M$ is a brick, then any non-zero endomorphism of $M$ is an isomorphism. Then the statement is trivial. Otherwise, there exists a non-zero non-isomorphic endomorphism $f:M\rightarrow M$. We take a $\Theta$-decomposition $(i_{f},X_{f},j_{f})$ of $f$. Since $M\in \mathcal{U}^{\perp}\cap \mathcal{T}$, we have $0\neq X_{f}\in\mathcal{U}^{\perp}\cap \mathcal{T}$. Note that $\Theta(X_{f})<\Theta(M)$ since $f$ is not an isomorphism. By induction hypothesis, there exists a non-zero endomorphism $g:X_{f}\rightarrow X_{f}$ admitting a $\Theta$-decomposition $(i_{g},X_{g},j_{g})$ with $X_{g}\in{\rm brick }(\mathcal{U}^{\perp}\cap \mathcal{T})$. Thus, we have a  $\Theta$-decomposition $(i_{g}i_{f},X_{g},j_{g}j_{f})$ of $f$ with $X_{g}\in{\rm brick }(\mathcal{U}^{\perp}\cap \mathcal{T})$.
\end{proof}

\begin{lemma}\label{L-2-4} Let $\mathcal{U}\subseteq\mathcal{T}$ be two torsion classes in $\mathcal{A}$. Then we have $$\mathcal{U}^{\perp}\cap \mathcal{T}=\Filt_{\Theta}( {\rm brick }(\mathcal{U}^{\perp}\cap \mathcal{T})).$$
\end{lemma}
\begin{proof}  It suffices to prove $\mathcal{U}^{\perp}\cap \mathcal{T}\subseteq\Filt_{\Theta}( {\rm brick }(\mathcal{U}^{\perp}\cap \mathcal{T}))$. For any non-zero indecomposable object $M\in\mathcal{U}^{\perp}\cap \mathcal{T}$,  by Lemma \ref{L-2-3}, there exists a non-zero endomorphism $f:M\rightarrow M$ and a $\Theta$-decomposition $(i_{f},X_{f},j_{f})$ with $X_{f}\in{\rm brick }(\mathcal{U}^{\perp}\cap \mathcal{T})$. In addition, we have a non-split stable $\mathbb{E}$-triangle
 $$X_{f}\stackrel{j_{f}}\rightarrowtail M\stackrel{g}\twoheadrightarrow M'\stackrel{\delta}\dashrightarrow.$$
If $\Theta(M')=0$, then $M\cong X_{f}\in{\rm brick }(\mathcal{U}^{\perp}\cap \mathcal{T})$. So let us assume $\Theta(M')>0$. Since $X_{f}$ is a brick and $i_{f}\neq0$, there exists a non-zero morphism $s:M'\rightarrow X_{f}$ such that $sg=i_{f}$. This implies that $M'\notin\mathcal{U}$. Using Lemma \ref{L-2-1} together with Lemma \ref{L-1}, we have the following commutative diagram
\begin{equation*}
\xymatrix{
 X_{f}~\ar@{=}[d] \ar@{>->}[r] & K \vphantom{\big|}  \ar@{>->}[d]^{} \ar@{->>}[r]^{} & U\vphantom{\big|}   \ar@{>->}[d]\ar@{-->}[r]^-{}& \\
  X_{f}~\ar@{>->}[r]^{j_{f}} & M\ar@{->>}[d]^{} \ar@{->>}[r]^{} & M' \ar@{->>}[d] \ar@{-->}[r]^-{}&\\
  &  V\ar@{=}[r]\ar@{-->}[d]^-{} & V\ar@{-->}[d]^-{} \\
   &  & }
\end{equation*}
where $U\in \mathcal{U}$ and $V\in \mathcal{U}^{\perp}$. Since $M\in \mathcal{U}^{\perp}\cap \mathcal{T}$ and $X_{f},U\in \mathcal{T}$, we have $K\in \mathcal{U}^{\perp}\cap\mathcal{T}$ and $V\in \mathcal{T}$. Note that, $\Theta(V)<\Theta(M)$ since $j_{f}\neq0$. If  $\Theta(K)=\Theta(M)$, then $K\cong M$ and thus $M'\cong U\in \mathcal{U}$, this is a contradiction. This implies that $0<\Theta(V)<\Theta(M)$. By induction hypothesis, we get $K,V\in\Filt_{\Theta}( {\rm brick }(\mathcal{U}^{\perp}\cap \mathcal{T}))$ and so is $M$.
\end{proof}

\begin{lemma}\label{L-4-18}  Let $[\mathcal{U},\mathcal{T}]$ be an interval in $\tors_{\Theta}(\mathcal{A})$.

$(1)$  We have $\mathcal{T}=\mathcal{U}\ast (\mathcal{U}^{\perp}\cap \mathcal{T})$.

$(2)$  Let $S\in{\rm brick }(\mathcal{U}^{\perp}\cap \mathcal{T})$ such that $\Theta(S)\leq\Theta(S')$ for any $S'\in{\rm brick }(\mathcal{U}^{\perp}\cap \mathcal{T})$. Then $\mathcal{T}\cap{^{\perp}}S$ is a torsion class.
\end{lemma}
\begin{proof} $(1)$ It is enough to show that $\mathcal{T}\subseteq\mathcal{U}\ast (\mathcal{U}^{\perp}\cap \mathcal{T})$. For any $M\in\mathcal{T}$, by Lemma \ref{L-2-1}, there exists a stable $\mathbb{E}$-triangle $U\stackrel{}\rightarrowtail M\twoheadrightarrow U'\stackrel{}\dashrightarrow$ such that $U\in \mathcal{U}$ and $U'\in\mathcal{U}^{\perp}$. Since $\mathcal{T}$ is a torsion class, we get $U'\in\mathcal{U}^{\perp}\cap \mathcal{T}$.

$(2)$ It is obvious that $\mathcal{T}\cap{^{\perp}}S$ is closed under extensions. We take a $\Theta$-deflation $f:M\twoheadrightarrow N$ such that $M\in\mathcal{T}\cap{^{\perp}}S$. It suffices to prove $N\in\mathcal{T}\cap{^{\perp}}S$. Note that $N\in\mathcal{T}$ since $\mathcal{T}$ is a torsion class. If $\Theta(N)=1$, then either $N\in \mathcal{U}$ or $N\in \mathcal{U}^{\perp}$ by Lemma \ref{L-2-1}.  For the former, we have $N\in \mathcal{U}\subseteq {^{\perp}}S$. For the latter, we  have $N\in{\rm brick }(\mathcal{U}^{\perp}\cap \mathcal{T})$. Since $\Theta(S)\leq\Theta(N)=1$, we have $S\in\Theta_{1}$. By Lemma \ref{L-4-18}, $^{\perp}S$ is a torsion class and thus $N\in{^{\perp}}S$. For the general case, let $g:N\rightarrow S$ be a morphism in $\mathcal{A}$. We take a $\Theta$-decomposition $(i_{g},X_{g},j_{g})$ for $g$. If $\Theta(N)=\Theta(X_{g})$, then $g\cong j_{g}$ is a $\Theta$-inflation. Since $S\in \mathcal{U}^{\perp}$, we have $N\in\mathcal{U}^{\perp}\cap \mathcal{T}$. By Lemma \ref{L-2-3}, there exists a brick $S'\in{\rm brick }(\mathcal{U}^{\perp}\cap \mathcal{T})$ such that $\Theta(S')\leq\Theta(N)$. Since $\Theta(S')\leq\Theta(N)\leq\Theta(S)\leq\Theta(S')$, we have $N\cong S$ and hence $f=0$, this is a contradiction. This implies $\Theta(X_{g})<\Theta(N)$. By induction hypothesis, we have $X_{g}\in\mathcal{T}\cap{^{\perp}}S$. In this case, we get $j_{g}=0$ and then $g=0$. We conclude that $N\in\mathcal{T}\cap{^{\perp}}S$.
\end{proof}

\begin{lemma}\label{L-4-166}  If $M\in\Theta_{1}$, then $^{\perp}M$ is a torsion class.
\end{lemma}
\begin{proof}  Clearly, it suffices to show that $\Fac_{\Theta}(^{\perp}M)\subseteq{^{\perp}M}$. To see this, we take a $\Theta$-deflation $f:P\twoheadrightarrow Q$ with $P\in{^\perp M}$. If $Q\notin{^{\perp}M}$, then there exists a  $\Theta$-deflation $g:Q\twoheadrightarrow M$ by Lemma \ref{L-3-77}(1). Since $P\in{^\perp M}$, we get $fg=0$ and thus $M\cong0$, this is a contradiction. This implies  $Q\in{^{\perp}M}$ and hence  $^{\perp}M$ is a torsion class.
\end{proof}

\begin{remark}For any $S\in \mathcal{A}$, $^{\perp}S$ is not a torsion class in general.  For instance, consider the length category $(D^{b}(\Lambda),l_{\mathcal{X}})$ as in  Example \ref{E-3-26}. We get easily $P_{1}[-1]\in {^{\perp}P_{2}}$ and $I_{2}[-1]\in\Fac_{l_{\mathcal{X}}}(P_{1}[-1])$. However, $I_{2}[-1]\notin {^{\perp}P_{2}}$ and hence $ {^{\perp}P_{2}}\notin\tors_{l_{\mathcal{X}}}(D^{b}(\Lambda))$.
\end{remark}

For a set $\mathcal{X}$, we denote by $\#\mathcal{X}$ its cardinality. Now we can state our first main result of this subsection.

\begin{theorem}\label{main4} Let $(\mathcal{A},\Theta)$ be a length category. Suppose that $[\mathcal{U},\mathcal{T}]\subseteq\tors_{\Theta}(\mathcal{A})$.

$(1)$ We have $\#[\mathcal{U},\mathcal{T}]=1$ if and only if {\rm brick}$(\mathcal{U}^{\perp}\cap \mathcal{T})=\emptyset$.

$(2)$ If $\#${\rm brick}$(\mathcal{U}^{\perp}\cap \mathcal{T})=1$, then $\#[\mathcal{U},\mathcal{T}]=2$.

$(3)$ If $\#[\mathcal{U},\mathcal{T}]=2$, then there exists a brick $S\in{\rm brick }(\mathcal{U}^{\perp}\cap \mathcal{T})$ satisfying the following conditions:
\begin{itemize}
   \item [(i)]$\Theta(S)\leq\Theta(S')$ for any $S'\in{\rm brick }(\mathcal{U}^{\perp}\cap \mathcal{T})$.
   \item [(ii)] $\mathcal{T}=\mathrm{T}_{\Theta}(\mathcal{U}\cup  S)$ and $\mathcal{U}=\mathcal{T}\cap{^{\perp}}S$.
   \item [(iii)] $S$ satisfies $$S\in\bigcap_{S'\in{\rm brick }(\mathcal{U}^{\perp}\cap \mathcal{T})}(\Sub_{\Theta}(S')\cap\Fac_{\Theta}(S')).$$
  \end{itemize}
Moreover, the brick $S$ is unique up to isomorphism.
\end{theorem}
\begin{proof}
(1) If $\#[\mathcal{U},\mathcal{T}]=1$, then {\rm brick}$(\mathcal{U}^{\perp}\cap \mathcal{T})=\emptyset$ by Lemma \ref{L-2-4}. Conversely, suppose that {\rm brick }$(\mathcal{U}^{\perp}\cap \mathcal{T})=\emptyset$. If $\mathcal{U}\subsetneq\mathcal{T}$, then there exists a non-zero object $X\in\mathcal{U}^{\perp}\cap \mathcal{T}$. Then by Lemma \ref{L-2-3} there is a brick $X'\in{\rm brick }(\mathcal{U}^{\perp}\cap \mathcal{T})$, this is a contradiction.

$(2)$ Suppose that $\#${\rm brick }$(\mathcal{U}^{\perp}\cap \mathcal{T})=1$ and $\mathcal{U}\subseteq \mathcal{V}\subseteq \mathcal{T}$. Clearly, $\#{\rm brick }(\mathcal{U}^{\perp}\cap \mathcal{V})\leq 1$. If $\#{\rm brick }(\mathcal{U}^{\perp}\cap \mathcal{V})=0$, then $\mathcal{U}=\mathcal{V}$ by (1). If $\#{\rm brick }(\mathcal{U}^{\perp}\cap \mathcal{V})=1$, then $\mathcal{U}^{\perp}\cap \mathcal{V}=\mathcal{U}^{\perp}\cap \mathcal{T}$  by Lemma \ref{L-2-4}. By using Lemma \ref{L-4-18}(1), we have $\mathcal{V}=\mathcal{U}\ast (\mathcal{U}^{\perp}\cap \mathcal{V})=\mathcal{T}$. This shows that $\#[\mathcal{U},\mathcal{T}]=2$.

$(3)$ For any $S\in{\rm brick }(\mathcal{U}^{\perp}\cap \mathcal{T})$, we set $\mathcal{T}_{S}:=\mathrm{T}_{\Theta}(\mathcal{U}\cup S)$ and $\mathcal{U}_{S}=\mathcal{T}\cap\mathrm{T}_{\Theta}(^{\perp}S)$. Then one can check that $\mathcal{U}\subsetneq\mathcal{T}_{S}\subseteq \mathcal{T}$ and $\mathcal{U}\subseteq \mathcal{U}_{S}\subseteq\mathcal{T}_{S}$.

We can take $S\in{\rm brick }(\mathcal{U}^{\perp}\cap \mathcal{T})$ such that $\Theta(S)\leq\Theta(S')$ for any  $S'\in{\rm brick }(\mathcal{U}^{\perp}\cap \mathcal{T})$. Since $\#[\mathcal{U},\mathcal{T}]=2$, we have $\mathcal{T}=\mathcal{T}_{\mathcal{S}}$. Thus, we get $\mathcal{U}=\mathcal{T}\cap{^{\perp}}S$ by Lemma \ref{L-4-18}. So we have shown (i) and (ii). For (iii), we take $S'\in{\rm brick }(\mathcal{U}^{\perp}\cap \mathcal{T})$. Observe that $\mathcal{T}=\mathcal{T}_{\mathcal{S'}}$. If $S'\in{^{\perp}}S$, then $\mathcal{T}=\mathcal{T}_{\mathcal{S}'}=\mathrm{T}_{\Theta}(\mathcal{U}\cup S')\subseteq\mathcal{T}\cap{^{\perp}}S=\mathcal{U}$, this is a contradiction. Thus, there exists a non-zero morphism $f:S'\rightarrow S$ in $\mathcal{A}$. We take a $\Theta$-decomposition $(i_{f},X_{f},j_{f})$ of $f$. Clearly, $X_{f}\in\mathcal{U}^{\perp}\cap \mathcal{T}$. By Lemma \ref{L-2-3}, there exists a brick $S''\in{\rm brick }(\mathcal{U}^{\perp}\cap \mathcal{T})$ such that $\Theta(S'')\leq\Theta(X_{f})\leq\Theta(S)$.  The preceding argument implies $\Theta(X_{f})=\Theta(S)$ and hence $j_{f}$ is an isomorphism. Then $f\cong i_{f}$ and thus $S\in\Fac_{\Theta}(S')$. In particular, if $\Theta(S')=\Theta(S)$, then $f$ is an isomorphism. This means that $S$ is unique up to isomorphism. By symmetry, there exists a non-zero morphism $g:S\rightarrow S'$ with a $\Theta$-decomposition  $(i_{g},X_{g},j_{g})$. By using Lemma \ref{L-2-3}, we have $S\cong X_{g}\in\Sub_{\Theta}(S')$.
\end{proof}

Theorem \ref{main4}(3) allows us to introduce brick labelling in the Hasse quiver of $\tors_{\Theta}(\mathcal{A})$, which first appeared in \cite{De}.

\begin{definition}\label{brick} Let $q:\mathcal{T}\rightarrow \mathcal{U}$ be an arrow in {\rm Hasse}($\tors(\mathcal{A})$). The {\em brick label} of $q$ is the unique brick $S\in{\rm brick }(\mathcal{U}^{\perp}\cap \mathcal{T})$ from Theorem \ref{main4}(3). In this case, we write $\mathcal{T}\stackrel{S}\rightarrow \mathcal{U}$.
\end{definition}

Dually, for an arrow  $p:\mathcal{V}\rightarrow \mathcal{F}$  in {\rm Hasse}($\torf_{\Theta}(\mathcal{A})$),  the {brick label} of $p$ is the  unique brick  $S\in{\rm brick }(^{\perp}\mathcal{F}\cap \mathcal{V})$ (up to isomorphism) such that $\Theta(S)\leq\Theta(S')$ for any $S'\in{\rm brick }(^{\perp}\mathcal{F}\cap \mathcal{V})$. In this case, we write $\mathcal{V}\stackrel{S}\rightarrow \mathcal{F}$. We have the following relationship between $\tors_{\Theta}(\mathcal{A})$ and $\torf_{\Theta}(\mathcal{A})$ with respect to brick labelling.

\begin{proposition}\label{p-3-10} The isomorphism
 \begin{align*}
 \tors_{\Theta} (\mathcal{A})&\stackrel{}{\longrightarrow}\torf_{\Theta} (\mathcal{A})^{{\rm op}}\\
 \mathcal{T}&\longmapsto \mathcal{T}^{\perp}
 \end{align*}
 of complete lattices preserves brick labelling. That is, $\mathcal{T}\stackrel{S}\rightarrow \mathcal{U}$ is an arrow in {\rm Hasse}$(\tors_{\Theta}(\mathcal{A}))$ if and only if $\mathcal{U}^{\perp}\stackrel{S}\rightarrow \mathcal{T}^{\perp}$ is an arrow in {\rm Hasse}$(\torf_{\Theta}(\mathcal{A}))$.
\end{proposition}
\begin{proof} The isomorphism follows immediately from Theorem \ref{main2}(3). In particular,  $p:\mathcal{T}\stackrel{}\rightarrow \mathcal{U}$ is an arrow in {\rm Hasse}$(\tors_{\Theta}(\mathcal{A}))$ if and only if  $q:\mathcal{U}^{\perp}\stackrel{}\rightarrow \mathcal{T}^{\perp}$ is an arrow in {\rm Hasse}$(\torf_{\Theta}(\mathcal{A}))$. Note that  $\mathcal{U}^{\perp}\cap \mathcal{T}=\mathcal{U}^{\perp}\cap {^{\perp}}(\mathcal{T}^{\perp})$. Hence, the definition  of brick label yields $S_{p}=S_{q}$.
\end{proof}

The labelled Hasse quiver of $\tors_{\Theta} (\mathcal{A})$ has the following global structure:

\begin{lemma}\label{L-3-18} Let $X$ be an object in $\Theta_{1}$.

$(1)$ $\mathcal{A}\stackrel{X}\rightarrow{^{\perp}X}$ is an  arrow  in {\rm Hasse}$(\tors_{\Theta}(\mathcal{A}))$.

$(1)'$ $\mathcal{A}\stackrel{X}\rightarrow X{^{\perp}}$ is an  arrow  in {\rm Hasse}$(\torf_{\Theta}(\mathcal{A}))$.

$(2)$ $\add(X)\stackrel{X}\rightarrow 0$ is an  arrow  in {\rm Hasse}$(\tors_{\Theta}(\mathcal{A}))$.

$(2')$ $\add(X)\stackrel{X}\rightarrow 0$ is an  arrow  in {\rm Hasse}$(\torf_{\Theta}(\mathcal{A}))$.
\end{lemma}

\begin{proof} By Lemma \ref{L-4-166}, $^{\perp}X$ is a torsion class. Since ${\rm brick }((^{\perp}X)^{\perp}\cap\mathcal{A})=X$, $\mathcal{A}\stackrel{X}\rightarrow{^{\perp}X}$ is an  arrow  in {\rm Hasse}$(\tors_{\Theta}(\mathcal{A}))$ by Theorem \ref{main4}(2). Similarly, we can prove $(1)'$. By using Proposition \ref{p-3-10}, we get (2) and $(2)'$.
\end{proof}

\begin{proposition}\label{3-11} Let $(\mathcal{A},\Theta)$ be a length wide category. The arrows  in {\rm Hasse}$(\tors_{\Theta}(\mathcal{A}))$ starting at $\mathcal{A}$ are precisely $\mathcal{A}\stackrel{X}\rightarrow{^{\perp}X}$, where $X$ runs through $\Theta_{1}$. Similarly, the arrows  in {\rm Hasse}$(\tors_{\Theta}(\mathcal{A}))$ ending at $0$ are precisely $\add X\stackrel{X}\rightarrow0$, where $X$ runs through $\Theta_{1}$.
\end{proposition}
\begin{proof} Let $\mathcal{A}\stackrel{X}\rightarrow \mathcal{U}$ be an arrow in {\rm Hasse}$(\tors_{\Theta}(\mathcal{A}))$. Then Theorem \ref{main4}(3) implies that $\mathcal{U}=\mathcal{A}\cap{^{\perp}}X={^{\perp}}X$. For any $0\neq M\in {^{\perp}}X$, there exists a $\Theta$-deflation $M\twoheadrightarrow S$ for some $S\in\Theta_{1}$.  Since ${^{\perp}}X$ is a torsion class, we have $S\in{^{\perp}}X$. By Theorem \ref{main4}(3), this implies $S\cong X$ since $\Theta(X)\leq\Theta(S)$.  Thus, the  first statement follows  from Lemma  \ref{L-3-18}(1).

Now, let $\mathcal{T}\stackrel{X}\rightarrow 0$ be an arrow in {\rm Hasse}$(\tors_{\Theta}(\mathcal{A}))$. Then Theorem \ref{main4}(3) implies that $\mathcal{T}=\mathrm{T}_{\Theta}(S)$. By Proposition \ref{P-2-6}, there exists a $\Theta$-deflation $X\twoheadrightarrow S$ for some $S\in\Theta_{1}$. Since $S\in\mathcal{T}$, we have $\Theta(X)\leq\Theta(S)$ and hence $X\cong S$. One can check that $\mathcal{T}=\add X$ by Lemma \ref{L-3-18}(2).  By Lemma \ref{L-3-18}(2) again, we obtain the second statement.
\end{proof}

\begin{remark} Let $(\mathcal{A},\Theta)$ be a length category and set $\Theta'=l_{\Theta_{\infty}}$. Then $(\mathcal{A},\Theta')$ is a length wide category by Remark \ref{R-1}. Hence, we can apply Proposition \ref{3-11} to $(\mathcal{A},\Theta')$.
\end{remark}

To state our second main result, we introduce the following definition.

\begin{definition}\label{SF}  We say that a length category  $(\mathcal{A},\Theta)$ is {\em standard} if for any brick $S$ in $\mathcal{A}$, we have $\Sub_{\Theta}(S)\cap\Fac_{\Theta}(S)=\{0,S\}$.
 \end{definition}

\begin{example}\label{E-4-25} $(1)$ If $\mathcal{A}$ is an exact length category, by \cite[Exercise 8.6]{Bu},
we obtain that $\mathcal{A}$ is an abelian length category.
One can easily check that $(\mathcal{A},l_{\rm sim(\mathcal{A})})$ is standard.

$(2)$ Let $\Lambda$ be a path algebra associated to a quiver of type $A_{n}$.  Suppose that $M,N\in\ind(D^{b}(\Lambda))$ such that $\Hom_{D^{b}(\Lambda)}(M,N)\neq0$ and $\Hom_{D^{b}(\Lambda)}(N,M)\neq0$. We claim that $M\cong N$. Note that $D^{b}(\Lambda)$ is independent of the choice of orientation of $A_{n}$. We only need to consider the case $1\rightarrow 2\rightarrow\cdots\rightarrow n$. By hypothesis, there exists an integer $i$ such that $M\cong X[i]$ and $N\cong Y[i]$ for some indecomposable $\Lambda$-modules $X,Y$.
For $P\in\ind(\mod\Lambda)$, we set
$$H_{+}^{0}(P)=\{Q\in\ind(\mod \Lambda)~|~\Hom_{\Lambda}(P,Q)\neq0\}.$$
Since $M\in H_{+}^{0}(N)$ and  $N\in H_{+}^{0}(M)$, we obtain that $M\cong N$  by \cite[Lemma 3.1]{Ar}. Now, let $\mathcal{X}$ be a semibrick in $D^{b}(\Lambda)$. As a consequence of Claim, $(\Filt_{\mathcal{T}}(\mathcal{X}),l_{\mathcal{X}})$ is a  standard length category. In particular, $(D^{b}(\Lambda),l_{\mathcal{S}})$ is  standard, where $\mathcal{S}$ is a simple-minded system given in Example \ref{Derived}.
\end{example}

Let $[\mathcal{U},\mathcal{T}]$ be an interval in $\tors_{\Theta}(\mathcal{A})$. We define
$${\rm Mbrick}[\mathcal{U},\mathcal{T}]:=\{S\in{\rm brick }(\mathcal{U}^{\perp}\cap \mathcal{T})~|~\mathcal{T}\cap{^{\perp}}S\in\tors_{\Theta}(\mathcal{A})\} ,$$
$${\rm Jbrick}[\mathcal{U},\mathcal{T}]:=\{S\in{\rm brick }(\mathcal{U}^{\perp}\cap \mathcal{T})~|~\mathrm{T}_{\Theta}(\mathcal{U}\cup S)\cap{^{\perp}}S\in\tors_{\Theta}(\mathcal{A})\}.$$
In particular, we have ${\rm Mbrick}[0,\mathcal{A}]=\{S\in{\rm brick }(\mathcal{A})~|~{^{\perp}}S\in\tors_{\Theta}(\mathcal{A})\}$ and ${\rm Jbrick}[0,\mathcal{A}]=\{S\in{\rm brick } (\mathcal{A})~|~\mathrm{T}_{\Theta}(S)\cap{^{\perp}}S\in\tors_{\Theta}(\mathcal{A})\}$. Then clearly we have ${\rm Mbrick}[0,\mathcal{A}]\subseteq{\rm Jbrick}[0,\mathcal{A}]$.
\begin{remark} Suppose that $\mathcal{A}$ is an exact category. We can check that ${\rm Mbrick}[\mathcal{U},\mathcal{T}]={\rm Jbrick}[\mathcal{U},\mathcal{T}]={\rm brick }(\mathcal{U}^{\perp}\cap \mathcal{T})$. In particular, ${\rm Mbrick}[0,\mathcal{A}]={\rm Jbrick}[0,\mathcal{A}]={\rm brick }(\mathcal{A})$.
\end{remark}

Let $P=(P,\leq)$ be a partially ordered set. For any $a\in P$, we define $a_{\ast}=\bigvee\{b\in P~|~b<a\}$ and $a^{\ast}=\bigwedge\{b\in P~|~a> b\}$. An element $a\in P$ is {\em completely join-irreducible}  if $a=\bigvee L$ for some $L\subseteq P$ implies $a\in L$. In this case, $a_{\ast}$ is the unique element covered by $a$. If $L$ is a finite set, then the converse is true. One can check that $a\in P$ is completely join-irreducible if and only if $a\neq a_{\ast}$. In the same way, an element $a\in P$ is {\em completely meet-irreducible}  if $a=\bigwedge L$ for some $L\subseteq P$ implies $a\in L$, equivalently, $a\neq a^{\ast}$. We denote by {\rm j-irr}$^{\rm c}(P)$ (resp. {\rm m-irr}$^{\rm c}(P)$) the set of completely join-irreducible (resp. completely meet-irreducible) elements in $P$.

\begin{lemma}\label{3-13} Let $(\mathcal{A},\Theta)$ be a standard length category with $[\mathcal{U},\mathcal{T}]\subseteq\tors_{\Theta}(\mathcal{A})$.

$(1)$ Let $S$ be a brick in $\mathcal{U}^{\perp}$ and $S'\in\mathrm{T}_{\Theta}(\mathcal{U}\cup S)$. Then $S'\in{^{\perp}}S$ or $S\in\Fac_{\Theta}(S')$.

$(2)$ If $S\in{\rm Jbrick}[\mathcal{U},\mathcal{T}]$, then
 $$\mathcal{T}_{S}:=\mathrm{T}_{\Theta}(\mathcal{U}\cup S)\in{\rm j}\text{-}{\rm irr}^{\rm c}([\mathcal{U},\mathcal{T}])~\text{and}~\mathcal{T}_{S^\ast}=\mathcal{T}_{S}\cap{^{\perp}}S.$$
In particular, $\mathcal{T}_{S}\rightarrow\mathcal{T}_{S^\ast}$ is an arrow in {\rm Hasse}$([\mathcal{U},\mathcal{T}])$.

$(3)$ If $S\in{\rm Mbrick}[\mathcal{U},\mathcal{T}]$, then
 $$\mathcal{U}_{S}:=\mathcal{T}\cap{^{\perp}}S\in{\rm m}\text{-}{\rm irr}^{\rm c}([\mathcal{U},\mathcal{T}])~\text{and}~\mathcal{U}_{S^\ast}=\mathrm{T}_{\Theta}(\mathcal{U}_{S}\cup S).$$
In particular, $\mathcal{U}_{S^\ast}\rightarrow\mathcal{U}_{S}$ is an arrow in {\rm Hasse}$([\mathcal{U},\mathcal{T}])$.
\end{lemma}
\begin{proof} $(1)$ If $S'\in{^{\perp}S}$, then the statement is trivial, so we assume there exists a non-zero morphism $f:S'\rightarrow S$. We take a $\Theta$-decomposition $(i_{f},X_{f},j_{f} )$ for $f$. By Proposition \ref{P-2-6}, we can check that $\mathrm{T}_{\Theta}(\mathcal{U}\cup S)=\Filt_{\Theta}(\mathcal{U}\cup\Fac_{\Theta}(S))$. Since $X_{f}\in \mathcal{U}^{\perp}\cap \mathrm{T}_{\Theta}(\mathcal{U}\cup S)$, there exists a $\Theta$-inflation $g:M\rightarrowtail X_{f}$ with $M\in\Fac_{\Theta}(S)$. Then $j_{f}g$ is  a $\Theta$-inflation and hence $M\in\Fac_{\Theta}(S)\cap\Sub_{\Theta}(S)$. Since $(\mathcal{A},\Theta)$ is  standard, we have $M\cong S\cong X_{f}$ and then $j_{f}$ is an isomorphism. This implies that $f\cong i_{f}$, thus $S\in\Fac_{\Theta}(S')$.

$(2)$ First of all, it is immediate that $\mathcal{U}\subseteq \mathrm{T}_{S}\cap{^{\perp}}S\subsetneq \mathrm{T}_{S}\subseteq \mathcal{T}$. Suppose that $\mathcal{V}\in[\mathcal{U},\mathcal{T}_{S}]$ with $\mathcal{V}\subsetneq\mathcal{T}_{S}$. If $\mathcal{V}\subseteq{^{\perp}}S$, then $\mathcal{V}\subseteq\mathrm{T}_{S}\cap{^{\perp}}S$. Otherwise,  there exists a non-zero morphism $f:S'\rightarrow S$ for some $S'\in\mathcal{V}$.  Thus  (1) implies that $S\in\Fac_{\Theta}(S')$ and then  $\mathcal{V}=\mathcal{T}_{S}$, this is a contradiction. This shows that $\mathcal{T}_{S^\ast}=\mathcal{T}_{S}\cap{^{\perp}}S$ and
 $\mathcal{T}_{S}\in{\rm j}\text{-}{\rm irr}^{\rm c}([\mathcal{U},\mathcal{T}])$. Similarly, we can prove (3).
\end{proof}
The second main result of this subsection is the following.
\begin{theorem}\label{main-N} Let $(\mathcal{A},\Theta)$ be a standard length category with $[\mathcal{U},\mathcal{T}]\subseteq\tors_{\Theta}(\mathcal{A})$.

$(1)$ We have $\#[\mathcal{U},\mathcal{T}]=2$ if and only if $\#${\rm brick }$(\mathcal{U}^{\perp}\cap \mathcal{T})=1$. In this case, there exists a unique brick $S\in{\rm brick }(\mathcal{U}^{\perp}\cap \mathcal{T})$ such that
$$\mathcal{T}=\mathrm{T}_{\Theta}(\mathcal{U}\cup  S),~\mathcal{U}=\mathcal{T}\cap {^{\perp}}S~\text{\rm and}~\mathcal{U}^{\perp}\cap \mathcal{T}=\Filt_{\Theta}(S).$$
In particular, $\mathcal{T}\stackrel{X}\rightarrow \mathcal{U}$ is an arrow in {\rm Hasse}$([\mathcal{U},\mathcal{T}])$.

$(2)$ There exists a bijection
$$ \mathcal{T}_{?}:{\rm Jbrick}[\mathcal{U},\mathcal{T}]\longrightarrow {\rm j}\text{-}{\rm irr}^{\rm c}([\mathcal{U},\mathcal{T}]);~S\longmapsto \mathrm{T}_{\Theta}(\mathcal{U}\cup S).$$
In this case, $\mathcal{T}_{S^\ast}=\mathcal{T}_{S}\cap{^{\perp}}S$ and $\mathcal{T}_{S}\stackrel{S}\rightarrow\mathcal{T}_{S^\ast}$ is the only arrow starting at $\mathcal{T}_{S}$.

$(3)$ There exists a bijection
$$ \mathcal{U}_{?}:{\rm Mbrick}[\mathcal{U},\mathcal{T}]\longrightarrow {\rm m}\text{-}{\rm irr}^{\rm c}([\mathcal{U},\mathcal{T}]);~S\longmapsto \mathcal{T}\cap{^{\perp}}S.$$
In this case, $\mathcal{U}_{S^\ast}=\mathrm{T}_{\Theta}(\mathcal{U}_{S}\cup S)$ and $\mathcal{U}_{S^\ast}\stackrel{S}\rightarrow\mathcal{U}_{S}$ is the only arrow ending at $\mathcal{U}_{S}$.

$(4)$ There exists two bijections
$$ {\rm Jbrick}[0,\mathcal{A}]\longrightarrow {\rm j}\text{-}{\rm irr}^{\rm c}(\tors_{\Theta}(\mathcal{A}));~S\longmapsto \mathrm{T}_{\Theta}(S),$$
$${\rm Mbrick}[0,\mathcal{A}]\longrightarrow {\rm m}\text{-}{\rm irr}^{\rm c}(\tors_{\Theta}(\mathcal{A}));~S\longmapsto {^{\perp}}S.$$
In particular, ${\rm m}\text{-}{\rm irr}^{\rm c}(\tors_{\Theta}(\mathcal{A}))$ is a subset of ${\rm j}\text{-}{\rm irr}^{\rm c}(\tors_{\Theta}(\mathcal{A})).$
\end{theorem}
\begin{proof} (1) This follows immediately from  Lemma \ref{L-2-4} and  Theorem \ref{main4}.

$(2)$ Firstly, we prove that the map $\mathcal{T}_{?}:{\rm Jbrick}[\mathcal{U},\mathcal{T}]\rightarrow {\rm j}\text{-}{\rm irr}^{\rm c}([\mathcal{U},\mathcal{T}])$ is well-defined. Take any $S\in{\rm Jbrick}[\mathcal{U},\mathcal{T}]$. By Lemma \ref{3-13}, we have
 $\mathcal{T}_{S}\in{\rm j}\text{-}{\rm irr}^{\rm c}([\mathcal{U},\mathcal{T}])~\text{and}~\mathcal{T}_{S^\ast}=\mathcal{T}_{S}\cap{^{\perp}}S.$
Suppose that there  exists a brick $S'\in {\rm Jbrick}[\mathcal{U},\mathcal{T}]$ such that $\mathcal{T}_{S'}=\mathcal{T}_{S}$. If $S'\in{^{\perp}S}$, then $\mathcal{T}_{S'}\subseteq \mathcal{T}_{S^\ast}\subsetneq\mathcal{T}_{S}$, this is a contradiction. Then there exists a non-zero morphism $S'\rightarrow S$. Since $S'\in\mathcal{T}_{S}$, we get $S\in\Fac_{\Theta}(S')$ by Lemma \ref{3-13}(1). By symmetry, we also have $S'\in\Fac_{\Theta}(S)$ and hence $S\cong S'$.

Next we show that  $\mathcal{T}_{?}$ is inverse.  For any $\mathcal{V}\in${\rm j-irr}$^{\rm c}([\mathcal{U},\mathcal{T}])$, we have $\#[\mathcal{V}_{\ast},\mathcal{V}]=2$. By (1),  there exists a unique brick $S\in{\rm brick }(\mathcal{V}_{\ast}^{\perp}\cap \mathcal{V})$ such that
$$\mathcal{V}=\mathrm{T}_{\Theta}(\mathcal{V}_{\ast}\cup S),~\mathcal{V}_{\ast}=\mathcal{V}\cap {^{\perp}}S~\text{\rm and}~\mathcal{V}_{\ast}^{\perp}\cap \mathcal{V}=\Filt_{\mathcal{A}}(S).$$
Note that $\mathrm{T}_{\Theta}(\mathcal{U}\cup S)\subseteq \mathcal{V}$ and  $\mathrm{T}_{\Theta}(\mathcal{U}\cup S)\nsubseteq\mathcal{V}_{\ast}$.  This implies $\mathcal{V}=\mathrm{T}_{\Theta}(\mathcal{U}\cup S)$ and hence $\mathrm{T}_{\Theta}(\mathcal{U}\cup S)\cap {^{\perp}}S=\mathcal{V}_{\ast}$ is a torsion class. It follows that  $S\in{\rm Jbrick}[\mathcal{U},\mathcal{T}]$. Then we can define a map  $$\mathbf{L}:{\rm j}\text{-}{\rm irr}^{\rm c}([\mathcal{U},\mathcal{T}])\longrightarrow {\rm Jbrick}[\mathcal{U},\mathcal{T}];~\mathcal{V}\longmapsto S.$$
Suppose that there exists a torsion class $\mathcal{V}'\in${\rm j-irr}$^{\rm c}([\mathcal{U},\mathcal{T}])$ such that $\mathbf{L}(\mathcal{V}')=S$. Then $\mathcal{V}=\mathrm{T}_{\Theta}(\mathcal{U}\cup S)=\mathcal{V'}$. Hence, the map $\mathbf{L}$ is well-defined.
For any $S\in{\rm Jbrick}[\mathcal{U},\mathcal{T}]$, there exists an arrow $\mathcal{T}_{S}\stackrel{S}\rightarrow \mathcal{T}_{S^\ast}$ in  {\rm Hasse}$([\mathcal{U},\mathcal{T}])$ such that ${\rm brick }(\mathcal{T}_{S^\ast}^{\perp}\cap \mathcal{T}_{S})=S$. This proves $\mathbf{L}\mathcal{T}_{?}=\id$.  Conversely, we assume that $\mathcal{V}\in${\rm j-irr}$^{\rm c}([\mathcal{U},\mathcal{T}])$. Then ${\rm brick }(\mathcal{V}_{\ast}^{\perp}\cap \mathcal{V})=\{\mathbf{L}(\mathcal{V})\}$ and hence $\mathcal{T}_{\mathbf{L}(\mathcal{V})}=\mathrm{T}_{\Theta}(\mathcal{U}\cup \mathbf{L}(V))\subseteq\mathcal{V}$. Since $\mathrm{T}_{\Theta}(\mathcal{U}\cup \mathbf{L}(\mathcal{V}))\nsubseteq \mathcal{V}\cap {^{\perp}}\mathbf{L}(\mathcal{V})=\mathcal{V}_{\ast}$, we conclude that  $\mathcal{T}_{\mathbf{L}(\mathcal{V})}=\mathcal{V}$. This proves $\mathcal{T}_{?}\mathbf{L}=\id$.

$(3)$ Using the analogous arguments  used in (2), we can prove $\mathcal{U}_{?}$ is a bijection.

$(4)$ It follows immediately from (2) and (3).
\end{proof}

\begin{remark} This is a generalization of the classical result \cite[Theorems 3.3, 3.4]{De} of Demonet-Iyama-Reading-Reiten-Thomas for a finitely generated module category $\mathcal{A}$.
\end{remark}

We finish this subsection with a straightforward example illustrating Theorem \ref{main-N}.
\begin{example}\label{ED}  Keep the notation used in Example \ref{E-2-22}. By Example \ref{E-4-25}(2),  $(\mathcal{A},l_{\mathcal{Y}})$ is a standard length category. It is easy to check that ${\rm Mbrick}[0,\mathcal{A}]={\rm Jbrick}[0,\mathcal{A}]= {\rm brick}\mathcal{A}=\{P_{1},S_{1}[-1],S_{2}[-1],S_{3},P_{2},I_{2}[-1]\}.$ Let us list all 6 bricks, the corresponding completely meet-irreducible elements and completely join-irreducible elements as follows:
$$
\begin{tabular}{|p{1.1cm}|p{6.3cm}|p{7.8cm}|}
\hline
 brick& completely join-irreducible element&    completely meet-irreducible element\\
\hline
$P_{1}$&  $\add(P_{1})$  &   $\add(\{S_{2}[-1],S_{1}[-1],I_{2}[-1]\})={^{\perp}P_{1}}$ \\
\hline
$S_{1}[-1]$&    $\add(S_{1}[-1])$          &   $\add(\{S_{2}[-1],P_{1},P_{2},S_{3}\})={^{\perp}S_{1}}[-1]$  \\
\hline
$S_{2}[-1]$&       $\add(S_{2}[-1])$                &   $\add(\{S_{1}[-1],P_{1},P_{2},S_{3},I_{2}[-1]\})={^{\perp}S_{2}}[-1]$   \\
\hline
$S_{3}$&      $\add(\{S_{3},P_{2},P_{1}\})$        &  $\add(\{P_{2},P_{1}\})$        \\
\hline
$P_{2}$&      $\add(\{P_{2},P_{1}\})$        &   $\add(\{S_{2}[-1],P_{1}\})$  \\
\hline
$I_{2}[-1]$&      $\add(\{I_{2}[-1],S_{1}[-1]\})$            &   $\add(\{S_{1}[-1],P_{1},P_{2},P_{3}\})$    \\
\hline
\end{tabular}
$$
\tikzset{every picture/.style={line width=0.75pt}} 
The labelled Hasse quiver of $\tors_{l_{\mathcal{Y}}}(\mathcal{A})$ is given by
$$
\begin{tikzpicture}[x=0.75pt,y=0.75pt,yscale=-1,xscale=1]

\draw    (187.68,613.14) -- (309.02,685.76) ;
\draw [shift={(310.74,686.79)}, rotate = 210.9] [color={rgb, 255:red, 0; green, 0; blue, 0 }  ][line width=0.75]    (10.93,-3.29) .. controls (6.95,-1.4) and (3.31,-0.3) .. (0,0) .. controls (3.31,0.3) and (6.95,1.4) .. (10.93,3.29)   ;
\draw    (319.22,609.16) -- (319.19,682.73) ;
\draw [shift={(319.18,684.73)}, rotate = 270.02] [color={rgb, 255:red, 0; green, 0; blue, 0 }  ][line width=0.75]    (10.93,-3.29) .. controls (6.95,-1.4) and (3.31,-0.3) .. (0,0) .. controls (3.31,0.3) and (6.95,1.4) .. (10.93,3.29)   ;
\draw    (464.29,614.14) -- (329.64,687.76) ;
\draw [shift={(327.89,688.72)}, rotate = 331.33] [color={rgb, 255:red, 0; green, 0; blue, 0 }  ][line width=0.75]    (10.93,-3.29) .. controls (6.95,-1.4) and (3.31,-0.3) .. (0,0) .. controls (3.31,0.3) and (6.95,1.4) .. (10.93,3.29)   ;
\draw    (187.77,514.7) -- (186.71,585.16) ;
\draw [shift={(186.68,587.16)}, rotate = 270.86] [color={rgb, 255:red, 0; green, 0; blue, 0 }  ][line width=0.75]    (10.93,-3.29) .. controls (6.95,-1.4) and (3.31,-0.3) .. (0,0) .. controls (3.31,0.3) and (6.95,1.4) .. (10.93,3.29)   ;
\draw    (226.73,508.2) -- (457.89,591.62) ;
\draw [shift={(459.77,592.3)}, rotate = 199.84] [color={rgb, 255:red, 0; green, 0; blue, 0 }  ][line width=0.75]    (10.93,-3.29) .. controls (6.95,-1.4) and (3.31,-0.3) .. (0,0) .. controls (3.31,0.3) and (6.95,1.4) .. (10.93,3.29)   ;
\draw    (60.48,377.9) -- (154.92,591.57) ;
\draw [shift={(155.73,593.4)}, rotate = 246.15] [color={rgb, 255:red, 0; green, 0; blue, 0 }  ][line width=0.75]    (10.93,-3.29) .. controls (6.95,-1.4) and (3.31,-0.3) .. (0,0) .. controls (3.31,0.3) and (6.95,1.4) .. (10.93,3.29)   ;
\draw    (198.48,383.9) -- (189.18,484.81) ;
\draw [shift={(189,486.8)}, rotate = 275.27] [color={rgb, 255:red, 0; green, 0; blue, 0 }  ][line width=0.75]    (10.93,-3.29) .. controls (6.95,-1.4) and (3.31,-0.3) .. (0,0) .. controls (3.31,0.3) and (6.95,1.4) .. (10.93,3.29)   ;
\draw    (191.48,219.3) -- (73.38,343.75) ;
\draw [shift={(72,345.2)}, rotate = 313.5] [color={rgb, 255:red, 0; green, 0; blue, 0 }  ][line width=0.75]    (10.93,-3.29) .. controls (6.95,-1.4) and (3.31,-0.3) .. (0,0) .. controls (3.31,0.3) and (6.95,1.4) .. (10.93,3.29)   ;
\draw    (204.48,226.3) -- (200.07,344.2) ;
\draw [shift={(200,346.2)}, rotate = 272.14] [color={rgb, 255:red, 0; green, 0; blue, 0 }  ][line width=0.75]    (10.93,-3.29) .. controls (6.95,-1.4) and (3.31,-0.3) .. (0,0) .. controls (3.31,0.3) and (6.95,1.4) .. (10.93,3.29)   ;
\draw    (282.48,462.5) -- (320.39,580.5) ;
\draw [shift={(321,582.4)}, rotate = 252.19] [color={rgb, 255:red, 0; green, 0; blue, 0 }  ][line width=0.75]    (10.93,-3.29) .. controls (6.95,-1.4) and (3.31,-0.3) .. (0,0) .. controls (3.31,0.3) and (6.95,1.4) .. (10.93,3.29)   ;
\draw    (77,378.8) -- (262.57,434.33) ;
\draw [shift={(264.48,434.9)}, rotate = 196.66] [color={rgb, 255:red, 0; green, 0; blue, 0 }  ][line width=0.75]    (10.93,-3.29) .. controls (6.95,-1.4) and (3.31,-0.3) .. (0,0) .. controls (3.31,0.3) and (6.95,1.4) .. (10.93,3.29)   ;
\draw    (398.48,452.9) -- (333.37,584.71) ;
\draw [shift={(332.48,586.5)}, rotate = 296.29] [color={rgb, 255:red, 0; green, 0; blue, 0 }  ][line width=0.75]    (10.93,-3.29) .. controls (6.95,-1.4) and (3.31,-0.3) .. (0,0) .. controls (3.31,0.3) and (6.95,1.4) .. (10.93,3.29)   ;
\draw    (406.48,352.9) -- (406.99,422.8) ;
\draw [shift={(407,424.8)}, rotate = 269.59] [color={rgb, 255:red, 0; green, 0; blue, 0 }  ][line width=0.75]    (10.93,-3.29) .. controls (6.95,-1.4) and (3.31,-0.3) .. (0,0) .. controls (3.31,0.3) and (6.95,1.4) .. (10.93,3.29)   ;
\draw    (384.48,270.3) -- (403.82,325.41) ;
\draw [shift={(404.48,327.3)}, rotate = 250.67] [color={rgb, 255:red, 0; green, 0; blue, 0 }  ][line width=0.75]    (10.93,-3.29) .. controls (6.95,-1.4) and (3.31,-0.3) .. (0,0) .. controls (3.31,0.3) and (6.95,1.4) .. (10.93,3.29)   ;
\draw    (214,221.2) -- (288.62,250.57) ;
\draw [shift={(290.48,251.3)}, rotate = 201.48] [color={rgb, 255:red, 0; green, 0; blue, 0 }  ][line width=0.75]    (10.93,-3.29) .. controls (6.95,-1.4) and (3.31,-0.3) .. (0,0) .. controls (3.31,0.3) and (6.95,1.4) .. (10.93,3.29)   ;
\draw    (506.22,527.7) -- (476.84,590.99) ;
\draw [shift={(476,592.8)}, rotate = 294.9] [color={rgb, 255:red, 0; green, 0; blue, 0 }  ][line width=0.75]    (10.93,-3.29) .. controls (6.95,-1.4) and (3.31,-0.3) .. (0,0) .. controls (3.31,0.3) and (6.95,1.4) .. (10.93,3.29)   ;
\draw    (461.22,446.8) -- (475.04,462.95) -- (506.92,500.18) ;
\draw [shift={(508.22,501.7)}, rotate = 229.43] [color={rgb, 255:red, 0; green, 0; blue, 0 }  ][line width=0.75]    (10.93,-3.29) .. controls (6.95,-1.4) and (3.31,-0.3) .. (0,0) .. controls (3.31,0.3) and (6.95,1.4) .. (10.93,3.29)   ;
\draw    (549.22,430.7) -- (538.44,455.73) -- (519.01,500.86) ;
\draw [shift={(518.22,502.7)}, rotate = 293.29] [color={rgb, 255:red, 0; green, 0; blue, 0 }  ][line width=0.75]    (10.93,-3.29) .. controls (6.95,-1.4) and (3.31,-0.3) .. (0,0) .. controls (3.31,0.3) and (6.95,1.4) .. (10.93,3.29)   ;
\draw    (278,372.2) -- (532.48,417.2) ;
\draw [shift={(534.45,417.55)}, rotate = 190.03] [color={rgb, 255:red, 0; green, 0; blue, 0 }  ][line width=0.75]    (10.93,-3.29) .. controls (6.95,-1.4) and (3.31,-0.3) .. (0,0) .. controls (3.31,0.3) and (6.95,1.4) .. (10.93,3.29)   ;
\draw    (472,347.2) -- (552.39,406.36) ;
\draw [shift={(554,407.55)}, rotate = 216.35] [color={rgb, 255:red, 0; green, 0; blue, 0 }  ][line width=0.75]    (10.93,-3.29) .. controls (6.95,-1.4) and (3.31,-0.3) .. (0,0) .. controls (3.31,0.3) and (6.95,1.4) .. (10.93,3.29)   ;

\draw (158.71,602.06) node [anchor=north west][inner sep=0.75pt]   [align=left] {};
\draw (312.48,690.14) node [anchor=north west][inner sep=0.75pt]    {$0$};
\draw (175.51,632.25) node [anchor=north west][inner sep=0.75pt]  [font=\scriptsize]  {$\textcolor[rgb]{0.82,0.01,0.11}{S}\textcolor[rgb]{0.82,0.01,0.11}{_{2}}\textcolor[rgb]{0.82,0.01,0.11}{[}\textcolor[rgb]{0.82,0.01,0.11}{-1}\textcolor[rgb]{0.82,0.01,0.11}{]}$};
\draw (322.56,619.31) node [anchor=north west][inner sep=0.75pt]  [font=\scriptsize]  {$\textcolor[rgb]{0.82,0.01,0.11}{S}\textcolor[rgb]{0.82,0.01,0.11}{_{1}}\textcolor[rgb]{0.82,0.01,0.11}{[}\textcolor[rgb]{0.82,0.01,0.11}{-1}\textcolor[rgb]{0.82,0.01,0.11}{]}$};
\draw (287.43,590.51) node [anchor=north west][inner sep=0.75pt]    {$S_{1}[ -1]$};
\draw (155.9,593.49) node [anchor=north west][inner sep=0.75pt]    {$S_{2}[ -1]$};
\draw (458.13,595.48) node [anchor=north west][inner sep=0.75pt]    {$P_{1}$};
\draw (148.73,494) node [anchor=north west][inner sep=0.75pt]    {$S_{2}[ -1] \ P_{1}$};
\draw (170.73,542.6) node [anchor=north west][inner sep=0.75pt]  [font=\scriptsize]  {$\textcolor[rgb]{0.82,0.01,0.11}{P}\textcolor[rgb]{0.82,0.01,0.11}{_{1}}$};
\draw (220.73,522.6) node [anchor=north west][inner sep=0.75pt]  [font=\scriptsize,color={rgb, 255:red, 208; green, 2; blue, 27 }  ,opacity=1 ]  {$S_{2}[ -1]$};
\draw (0,357.2) node [anchor=north west][inner sep=0.75pt]    {$S_{2}[ -1] \ S_{1}[ -1] \ I_{2}[ -1]$};
\draw (35,405.2) node [anchor=north west][inner sep=0.75pt]  [font=\scriptsize,color={rgb, 255:red, 208; green, 2; blue, 27 }  ,opacity=1 ]  {$S_{1}[ -1]$};
\draw (161,358.2) node [anchor=north west][inner sep=0.75pt]    {$S_{2}[ -1] \ P_{1} \ P_{2} \ S_{3}$};
\draw (174,416.2) node [anchor=north west][inner sep=0.75pt]  [font=\scriptsize]  {$\textcolor[rgb]{0.82,0.01,0.11}{P}\textcolor[rgb]{0.82,0.01,0.11}{_{2}}$};
\draw (196,199.6) node [anchor=north west][inner sep=0.75pt]    {$\mathcal{A}$};
\draw (142,238.6) node [anchor=north west][inner sep=0.75pt]  [font=\scriptsize]  {$\textcolor[rgb]{0.82,0.01,0.11}{P}\textcolor[rgb]{0.82,0.01,0.11}{_{1}}$};
\draw (206,252.6) node [anchor=north west][inner sep=0.75pt]  [font=\scriptsize]  {$\textcolor[rgb]{0.82,0.01,0.11}{S}\textcolor[rgb]{0.82,0.01,0.11}{_{1}}\textcolor[rgb]{0.82,0.01,0.11}{[}\textcolor[rgb]{0.82,0.01,0.11}{-1}\textcolor[rgb]{0.82,0.01,0.11}{]}$};
\draw (239,438.2) node [anchor=north west][inner sep=0.75pt]    {$S_{1}[ -1] \ I_{2}[ -1]$};
\draw (296,475.2) node [anchor=north west][inner sep=0.75pt]  [font=\scriptsize]  {$\textcolor[rgb]{0.82,0.01,0.11}{I}\textcolor[rgb]{0.82,0.01,0.11}{_{2}}\textcolor[rgb]{0.82,0.01,0.11}{[}\textcolor[rgb]{0.82,0.01,0.11}{-1}\textcolor[rgb]{0.82,0.01,0.11}{]}$};
\draw (95,397.2) node [anchor=north west][inner sep=0.75pt]  [font=\scriptsize]  {$\textcolor[rgb]{0.82,0.01,0.11}{S}\textcolor[rgb]{0.82,0.01,0.11}{_{2}}\textcolor[rgb]{0.82,0.01,0.11}{[}\textcolor[rgb]{0.82,0.01,0.11}{-1}\textcolor[rgb]{0.82,0.01,0.11}{]}$};
\draw (358,427.2) node [anchor=north west][inner sep=0.75pt]    {$S_{1}[ -1] \ P_{1} \ P_{2}$};
\draw (397,468.2) node [anchor=north west][inner sep=0.75pt]  [font=\scriptsize]  {$\textcolor[rgb]{0.82,0.01,0.11}{P}\textcolor[rgb]{0.82,0.01,0.11}{_{1}}$};
\draw (352,329.6) node [anchor=north west][inner sep=0.75pt]    {$S_{1}[ -1] \ P_{1} \ P_{2} \ S_{3}$};
\draw (389,362.2) node [anchor=north west][inner sep=0.75pt]  [font=\scriptsize]  {$\textcolor[rgb]{0.82,0.01,0.11}{S}\textcolor[rgb]{0.82,0.01,0.11}{_{3}}$};
\draw (296,243.6) node [anchor=north west][inner sep=0.75pt]    {$S_{1}[ -1] \ P_{1} \ P_{2} \ S_{3} \ I_{2}[ -1]$};
\draw (355,288.6) node [anchor=north west][inner sep=0.75pt]  [font=\scriptsize]  {$\textcolor[rgb]{0.82,0.01,0.11}{I}\textcolor[rgb]{0.82,0.01,0.11}{_{2}}\textcolor[rgb]{0.82,0.01,0.11}{[}\textcolor[rgb]{0.82,0.01,0.11}{-1}\textcolor[rgb]{0.82,0.01,0.11}{]}$};
\draw (246,222.6) node [anchor=north west][inner sep=0.75pt]  [font=\scriptsize]  {$\textcolor[rgb]{0.82,0.01,0.11}{S}\textcolor[rgb]{0.82,0.01,0.11}{_{2}}\textcolor[rgb]{0.82,0.01,0.11}{[}\textcolor[rgb]{0.82,0.01,0.11}{-1}\textcolor[rgb]{0.82,0.01,0.11}{]}$};
\draw (488,505.8) node [anchor=north west][inner sep=0.75pt]    {$P_{2} \ P_{1}$};
\draw (502.22,546.1) node [anchor=north west][inner sep=0.75pt]  [font=\scriptsize]  {$\textcolor[rgb]{0.82,0.01,0.11}{P}\textcolor[rgb]{0.82,0.01,0.11}{_{2}}$};
\draw (442.12,465.9) node [anchor=north west][inner sep=0.75pt]  [font=\scriptsize]  {$\textcolor[rgb]{0.82,0.01,0.11}{S}\textcolor[rgb]{0.82,0.01,0.11}{_{1}}\textcolor[rgb]{0.82,0.01,0.11}{[}\textcolor[rgb]{0.82,0.01,0.11}{-1}\textcolor[rgb]{0.82,0.01,0.11}{]}$};
\draw (449,630.4) node [anchor=north west][inner sep=0.75pt]  [font=\scriptsize]  {$\textcolor[rgb]{0.82,0.01,0.11}{P}\textcolor[rgb]{0.82,0.01,0.11}{_{1}}$};
\draw (122,765.8) node [anchor=north west][inner sep=0.75pt]    {$$};
\draw (54,763.4) node [anchor=north west][inner sep=0.75pt]  [font=\Huge] [align=left] {};
\draw (539,410.6) node [anchor=north west][inner sep=0.75pt]    {$S_{3} \ P_{2} \ P_{1}$};
\draw (545.83,446.62) node [anchor=north west][inner sep=0.75pt]  [font=\scriptsize,color={rgb, 255:red, 139; green, 87; blue, 42 }  ,opacity=1 ]  {$\textcolor[rgb]{0.82,0.01,0.11}{S}\textcolor[rgb]{0.82,0.01,0.11}{_{3}}$};
\draw (300,362.6) node [anchor=north west][inner sep=0.75pt]  [font=\scriptsize]  {$\textcolor[rgb]{0.82,0.01,0.11}{S}\textcolor[rgb]{0.82,0.01,0.11}{_{2}}\textcolor[rgb]{0.82,0.01,0.11}{[}\textcolor[rgb]{0.82,0.01,0.11}{-1}\textcolor[rgb]{0.82,0.01,0.11}{]}$};
\draw (513,362.6) node [anchor=north west][inner sep=0.75pt]  [font=\scriptsize]  {$\textcolor[rgb]{0.82,0.01,0.11}{S}\textcolor[rgb]{0.82,0.01,0.11}{_{1}}\textcolor[rgb]{0.82,0.01,0.11}{[}\textcolor[rgb]{0.82,0.01,0.11}{-1}\textcolor[rgb]{0.82,0.01,0.11}{]}$};
\draw (0,712.6) node [anchor=north west][inner sep=0.75pt]   [align=left] {where each torsion class  is represented by its non-zero indecomposable objects
(unless $0$ and $\mathcal{A}$).};
\end{tikzpicture}
 $$
\end{example}

\section{Support $\tau$-tilting subcategories in extriangulated length categories}
Our aim in this section is to establish a unified framework to study the  $\tau$-tilting theory in length categories. In this section, we always assume that $((\mathcal{A}, \mathbb{E},\mathfrak{s}),\Theta)$ is a length category.  We start with the following notations.

For a subcategory $\mathcal{T}$, we define
$$^{\perp_{\Theta}}\mathcal{T}=\{M\in\mathcal{A}~|~\mathbb{E}_{\Theta}(M,T)=0~\text{for any}~T\in \mathcal{T}\}.$$
Dually, we can define $\mathcal{T}^{\perp_{\Theta}}$. An object $P\in\mathcal{T}$ is called {\em $\Theta$-projective}  if $P\in{^{\perp_{\Theta}}}\mathcal{T}$. We denote the subcategory of $\Theta$-projective objects in $\mathcal{T}$ by $\mathcal{P}(\mathcal{T})$. Dually, the {\em $\Theta$-injective} objects are defined, and the full subcategory of $\Theta$-injective objects in $\mathcal{T}$ is denoted by $\mathcal{I}(\mathcal{T})$. We say $\mathcal{A}$ {\em has enough $\Theta$-projectives}  if for any $M\in \mathcal{A}$, there exists a $\Theta$-deflation $P\twoheadrightarrow M$ satisfying $P\in\mathcal{P}(\mathcal{A})$.  Dually, we define that $\mathscr{C}$ {\em has enough $\Theta$-injectives}.

\begin{remark} Recall that $((\mathcal{A}, \mathbb{E}_{\Theta},\mathfrak{s}_{\Theta}),\Theta)$ is a stable length category (see Proposition \ref{11-11}). It is easy to see that $P\in\mathcal{P}(\mathcal{A})$ (resp. $I\in\mathcal{I}(\mathcal{A})$) if and only if  $P$ (resp. $I$) is projective (resp. injective) in $(\mathcal{A}, \mathbb{E}_{\Theta},\mathfrak{s}_{\Theta})$.
 \end{remark}
We consider the following two conditions for a subcategory $\mathcal{T}$:
$$(\dag):\text{For any}~P\in\mathcal{P}(\mathcal{A}),~\text{there is a left}~\mathcal{T}\text{-approximation}~f:P\rightarrow T~\text{with}~T\in\mathcal{T}.$$
$$(\ddag):\text{For any}~P\in\mathcal{P}(\mathcal{A}),~\text{there is a left}~\mathcal{T}\text{-approximation}~f:P\rightarrow T~\text{with}~T\in\mathcal{P}(\mathcal{T}).$$
Then we can introduce the following two definitions.

\begin{definition}\label{d-5-2} We say that a  torsion class $\mathcal{T}$  is {\em support} if it satisfies the condition $(\ddag)$.
\end{definition}

\begin{definition}\label{d-5-11} A subcategory $\mathcal{T}$ of $\mathcal{A}$ is a {\em support $\tau$-tilting subcategory} if it satisfies:
\begin{itemize}
  \item [(1)] $\mathcal{T}=\mathcal{P}({\Fac_{\Theta}(\mathcal{T})})$,
  \item [(2)] $\mathcal{T}$ is  {\em $\tau$-rigid}, i.e. $\mathcal{T}\subseteq {^{\perp_{\Theta}}}{\Fac_{\Theta}(\mathcal{T})}$,
  \item [(3)] $\mathcal{T}$ satisfies the condition $(\dag)$.
\end{itemize}
We call an object $T\in\mathcal{A}$ is a  {\em support $\tau$-tilting object} if $\add(T)$ is a support $\tau$-tilting subcategory.
\end{definition}
 We denote by ${\rm s}\tau$-${\rm tilt(\mathcal{A}})$ (resp. ${\rm s}$-${\rm tors_{\Theta}(\mathcal{A})}$) the set of  support $\tau$-tilting subcategories (resp.  support torsion classes) in $\mathcal{A}$.

\begin{remark} When $\Theta$ is not stable, our notion of support
$\tau$-tilting subcategories differs from that in \cite[Definition 3.5]{ZW}.
Specifically, we only require that $\mathbb{E}_{\Theta}(\mathcal{T},{\Fac_{\Theta}(\mathcal{T})})=0$,
whereas \cite{ZW} imposes the stronger condition $\mathbb{E}(\mathcal{T},{{\rm Defl}(\mathcal{T})})=0$,
where ${\rm Defl}(\mathcal T)=\{F\in\A\mid \mbox{there exists a deflation}~
 U\xrightarrow{~g~}F~\mbox{for some}~U\in\mathcal T\}$.
However,  if $\Theta$ is stable, then support $\tau$-tilting subcategories defined in \cite{ZW} also qualify as support $\tau$-tilting subcategories in our framework (cf. \cite[Lemma 3.7]{ZW}).
\end{remark}

\begin{remark}\label{R-5-4} Let $\Lambda$ be a finite dimensional algebra. We set $\Theta=l_{{\rm sim({\rm mod} \Lambda)}}$.  As stated in Example \ref{E-3-10}, we regard $(\mod \Lambda,\Theta)$ as a stable length category. In this case, stable $\mathbb{E}$-triangles are precisely short exact sequences of $\mod \Lambda$.

$(1)$ Let $\mathcal{T}$ be a support torsion class. Then $\mathcal{T}$ is functorially finite by Proposition \ref{P-5-8} and \cite[Proposition 2.1]{Du}. Conversely, any functorially finite torsion class $\mathcal{F}$ is support by  \cite[Theorem 2.7]{Ad}) and \cite[Proposition 2.14]{Ja}. This shows that support torsion classes in module category are just functorially finite torsion classes.

$(2)$ Let $\mathcal{S}$ be a support $\tau$-tilting subcategory. By definition, there exists a right exact sequence
\begin{equation}\label{E-1}
\Lambda\stackrel{f}\rightarrow S_{1}\rightarrow S_{2}\rightarrow0
\end{equation}
such that $f$ is a left $\mathcal{S}$-approximation with $S_{1}\in\mathcal{S}$. We claim that $S_{2}\in\mathcal{S}$. To see this, we  take a $\Theta$-decomposition $(i_{f},\Im(f),j_{f})$ of $f$. Applying the functor $\Hom(-,\Fac_{\Theta}(\mathcal{S}))$ to the sequence (\ref{E-1}), we obtain an  exact sequence
$$\Hom(S_{1},-)|_{{\rm Fac}_{\Theta}(\mathcal{S})}\stackrel{f^{\ast}}\rightarrow \Hom(\Im(f),-)|_{{\rm Fac}_{\Theta}(\mathcal{S})}\rightarrow {\rm Ext}^{1}(S_{2},-)|_{{\rm Fac}_{\Theta}(\mathcal{S})}\stackrel{}\rightarrow{\rm Ext}^{1}(S_{1},-)|_{{\rm Fac}_{\Theta}(\mathcal{S})}.$$
Note that $f$ is  a left $\Fac_{\Theta}(\mathcal{S})$-approximation (see Lemma \ref{L-5-5}). One can check directly that $j_{f}$ is also a left $\Fac_{\Theta}(\mathcal{S})$-approximation. This implies that $f^{\ast}$ is an epimorphism. On the other hand, since $\mathcal{S}$ is $\tau$-rigid, we have  ${\rm Ext}^{1}(S_{1},-)|_{{\rm Fac}_{\Theta}(\mathcal{S})}=0$. Thus $S_{2}\in\mathcal{P}({\Fac_{\Theta}(\mathcal{S})})=\mathcal{S}$ and the claim follows. Thus  support $\tau$-tilting objects are precisely  support $\tau$-tilting $\Lambda$-modules by \cite[Theorem 2.7]{Ad} and \cite[Proposition 2.14]{Ja}.
\end{remark}
 We are ready to state the main theorem of this section.
\begin{theorem}\label{main5} Let $(\mathcal{A}, \Theta)$ be a length category with enough $\Theta$-projectives. There is a bijection
$${\rm s}\tau\text{-}{\rm tilt(\mathcal{A})}\longrightarrow{\rm s}\text{-}\tors_{\Theta}(\mathcal{A})$$
given by $\mathcal{S}\mapsto\Fac_{\Theta}(\mathcal{S})$ with the inverse $\mathcal{T}\mapsto \mathcal{P}(\mathcal{T})$.
\end{theorem}
Before proving Theorem \ref{main5}, we need some preparations.

\begin{lemma}\label{L-5-5} Let $\mathcal{T}$ be a $\tau$-rigid  subcategory of $\mathcal{A}$. If $f:P\rightarrow T$ is a left~$\mathcal{T}$-approximation for some $P\in\mathcal{P}(\mathcal{A})$, then $f$ is a left ~$\Fac_{\Theta}(\mathcal{T})$-approximation.
\end{lemma}
\begin{proof} Clearly, $T\in\Fac_{\Theta}(\mathcal{T})$. Suppose that $g:P\rightarrow X$ is a morphism in $\mathcal{A}$ such that $X\in\Fac_{\Theta}(\mathcal{T})$. By hypothesis, there exists a $\Theta$-deflation $h:T'\twoheadrightarrow X$ for some $T'\in\mathcal{T}$. Since $P\in\mathcal{P}(\mathcal{A})$, there exists a morphism $s:P\rightarrow T'$ such that $g=hs$. Since $f$ is a left $\mathcal{T}$-approximation, there exists a morphism $t:T\rightarrow T'$ such that $s=tf$. In particular, we have $g=htf$. This shows $f$ is a left $\Fac_{\Theta}(\mathcal{T})$-approximation.
\end{proof}

\begin{proposition}\label{P-5-7} $(1)$ Let $\mathcal{T}$ be a $\tau$-rigid  subcategory of $\mathcal{A}$. Then $\Fac_{\Theta}(\mathcal{T})\in\tors_{\Theta}(\mathcal{A})$.

$(2)$  Let $\mathcal{T}$ be a support $\tau$-tilting subcategory of $\mathcal{A}$. Then $\Fac_{\Theta}(\mathcal{T})\in{\rm s}\text{-}\tors_{\Theta}(\mathcal{A})$.
\end{proposition}
\begin{proof}  (1) Let $X_{1}\stackrel{}\rightarrowtail X\twoheadrightarrow X_{2}\stackrel{}\dashrightarrow$
be a stable $\mathbb{E}$-triangle with $X_{1},X_{2}\in\Fac_{\Theta}(\mathcal{T})$.  It suffices to prove $X\subseteq\Fac_{\Theta}(\mathcal{T})$. Since $X_{2}\in\Fac_{\Theta}(\mathcal{T})$, there exists a stable $\mathbb{E}$-triangle $K\stackrel{}\rightarrowtail M\twoheadrightarrow X_{2}\stackrel{}\dashrightarrow$ with $M\in\mathcal{T}$. Consider the following commutative diagram
$$\xymatrix{
   &  K\ar[d]_-{a} \ar@{=}[r]^{} &  K \vphantom{\big|} \ar@{>->}[d]^-{c} & \\
   X_{1} \ar@{=}[d]_{} \ar[r]^{} &L  \ar[d]_-{b} \ar[r]^{} & M\ar@{->>}[d]^{}  \ar@{-->}[r]^{\delta} &  \\
   X_{1}\ar@{>->}[r]^{} & X \ar@{->>}[r]^{} \ar@{-->}[d]^{}  & X_{2} \ar@{-->}[d]^{}  \ar@{-->}[r]^{} &.  \\
    &   & & }
$$
Since $c$ is a $\Theta$-inflation, so is $a$ by Lemma \ref{4-77}. Thus $b$ is a $\Theta$-deflation. By Lemma \ref{L-5-1}, $\delta\in\mathbb{E}_{\Theta}(M,X_{1})=0$ and then $L\cong X_{1}\oplus  M\in\Fac(\mathcal{T})$. Thus $X\in\Fac_{\Theta}(\mathcal{T})$ and the assertion  follows.

$(2)$ We  note that  Lemma \ref{L-5-5} implies that $\Fac_{\Theta}(\mathcal{T})$ satisfies the condition $(\ddag)$. Therefore, using (1), we deduce that  $\Fac_{\Theta}(\mathcal{T})\in{\rm s}\text{-}\tors_{\Theta}~\mathcal{A}$.
\end{proof}

\begin{proposition}\label{P-5-8} Suppose that  $(\mathcal{A}, \Theta)$ has enough $\Theta$-projectives. If $\mathcal{T}\in{\rm s}\text{-}\tors_{\Theta}(\mathcal{A})$, then  $\mathcal{T}=\Fac_{\Theta}(\mathcal{P}(\mathcal{T}))$. In particular, we have $\mathcal{P}(\mathcal{T})\in{\rm s}\tau$-${\rm tilt(\mathcal{A})}$.
\end{proposition}
\begin{proof}  The inclusion $\Fac_{\Theta}(\mathcal{P}(\mathcal{T}))\subseteq\mathcal{T}$ is obvious.  For each  $T\in \mathcal{T}$, there is a $\Theta$-deflation $f:P\twoheadrightarrow T$ with $P\in\mathcal{P}(\mathcal{A})$. Since $\mathcal{T}$ is support, there exists a left $\mathcal{T}$-approximation $g:P\rightarrow T'$ for some $T'\in\mathcal{P}(\mathcal{T})$. Therefore, there exists a morphism $h:T'\rightarrow T$ such that $f=hg$. By  Lemma \ref{4-77}, the morphism $h$ is a $\Theta$-deflation and hence $T\in\Fac_{\Theta}(\mathcal{P}(\mathcal{T}))$. Then one can check that $\mathcal{P}(\mathcal{T})\in{\rm s}\tau$-${\rm tilt~\mathcal{A}}$.
\end{proof}
Now we are ready to prove Theorem \ref{main5}.

\textbf{{Proof of Theorem \ref{main5}.}} For a support $\tau$-tilting subcategory $\mathcal{S}$ of $\mathcal{A}$, we have $\Fac_{\Theta}(\mathcal{S})\in{\rm s}\text{-}\tors_{\Theta}(\mathcal{A})$ by Proposition \ref{P-5-7}. Note that we have $\mathcal{S}=\mathcal{P}({\Fac_{\Theta}(\mathcal{S})})$. Conversely, for a given  support torsion class $\mathcal{T}$, we have $\mathcal{P}(\mathcal{T})\in{\rm s}\tau$-${\rm tilt(\mathcal{A})}$ and $\mathcal{T}=\Fac_{\Theta}(\mathcal{P}(\mathcal{T}))$ by Proposition \ref{P-5-8}.
\fin\\

 Let $\Lambda$ be a finite dimensional algebra. For a subcategory $\mathcal{T}$ of $\mod \Lambda$, we  denote by $\mathrm{P}(\mathcal{T})$  the direct sum of one copy of each of the indecomposable Ext-projective objects in $\mathcal{T}$ up to isomorphism. We clearly have $\mathcal{P}(\mathcal{T})=\add(\mathrm{P}(\mathcal{T}))$. As a special case of Theorem \ref{main5}, we recover the following well-known fact in classical $\tau$-tilting theory.

\begin{corollary} \text{\rm (\cite[Theorem 2.7]{Ad})} Let $\Lambda$ be a finite dimensional $k$-algebra over an algebraically closed field $k$. There is a bijection
$${\rm s}\tau\text{-}{\rm tilt~\Lambda}\longrightarrow{\rm f}\text{-}\tors~\Lambda$$
given by $M\mapsto\Fac(M)$ with the inverse $\mathcal{T}\mapsto \mathrm{P}(\mathcal{T})$.
\end{corollary}
\begin{proof}
 The assertion immediately follows from  Remark \ref{R-5-4} and Theorem \ref{main5}.
\end{proof}

For $\mathcal{P},\mathcal{Q}\in{\rm s}\tau\text{-}{\rm tilt~\mathcal{A}}$, we write $\mathcal{P}\leq\mathcal{Q}$ if $\Fac_{\Theta}(\mathcal{P})\subseteq\Fac_{\Theta}(\mathcal{Q})$. Then we have a partial order on ${\rm s}\tau\text{-}{\rm tilt~(\mathcal{A})}$.  Then the Hasse quiver of {\rm s}$\tau\text{-}{\rm tilt(\mathcal{A})}$ and its brick labelling have the following description.

\begin{corollary} We have that {\rm Hasse}$({\rm s}\tau\text{-}{\rm tilt(\mathcal{A})})\cong{\rm Hasse}({\rm s}\text{-}\tors_{\Theta}(\mathcal{A}))$ is a full subquiver of ${\rm Hasse}(\tors_{\Theta}(\mathcal{A}))$. In this case, the brick label of $\mathcal{P}\rightarrow\mathcal{Q}$ in ${\rm s}\tau\text{-}{\rm tilt~\mathcal{A}}$ is the same as the brick label of  $\Fac(\mathcal{P})\rightarrow\Fac(\mathcal{Q})$ in ${\rm Hasse}({\rm s}\text{-}\tors_{\Theta}(\mathcal{A}))$.
\end{corollary}
\begin{proof} This follows immediately from Theorem \ref{main5}.
\end{proof}
Dually, we can define support torsion-free classes and support $\tau^{-1}$-tilting subcategories as follows.

\begin{definition} We say that a  torsion-free class $\mathcal{F}$  is {\em support} if for any $I\in\mathcal{I}(\mathcal{A})$, there exists a right $\mathcal{F}$-approximation $F\rightarrow I$ with $F\in\Sub_{\Theta}(\mathcal{F})$.
\end{definition}

\begin{definition}\label{d-5-1} A subcategory $\mathcal{F}$ of $\mathcal{A}$ is a {\em support $\tau^{-1}$-tilting subcategory} if it satisfies:
\begin{itemize}
  \item [(1)] $\mathcal{F}=\mathcal{I}({\Sub_{\Theta}(\mathcal{F})})$,
  \item [(2)]  $\mathcal{T}$ is  {\em $\tau^{-1}$-rigid}, i.e. $\mathcal{T}\subseteq {\Sub(\mathcal{T})}{^{\perp_{\Theta}}}$.
  \item [(3)] For each $I\in\mathcal{I}(\mathcal{A})$, there exists a right $\mathcal{F}$-approximation $F\rightarrow I$ with $F\in\mathcal{F}$.
    \end{itemize}
We call an object $T\in\mathcal{A}$ is a {\em support $\tau^{-1}$-tilting object} if $\add(T)$ is a support $\tau^{-1}$-tilting subcategory.
\end{definition}

We denote by ${\rm s}\tau^{-1}$-${\rm tilt(\mathcal{A})}$ (resp. ${\rm s}$-${\rm torf_{\Theta}(\mathcal{A})}$) the set of support $\tau^{-1}$-tilting subcategories (resp. support torsion-free classes) in $\mathcal{A}$. We obtain the dual version of Theorem \ref{main5}.

\begin{theorem}\label{main6} Let $(\mathcal{A}, \Theta)$ be a length category with enough $\Theta$-injective objects. There is a bijection
$${\rm s}\tau^{-1}\text{-}{\rm tilt(\mathcal{A})}\longrightarrow{\rm s}\text{-}\torf_{\Theta}(\mathcal{A})$$
given by $\mathcal{S}\mapsto\Sub_{\Theta}(\mathcal{S})$ with the inverse $\mathcal{F}\mapsto \mathcal{I}(\mathcal{F})$.
\end{theorem}

Next we consider  when torsion classes and torsion-free classes are support. For this purpose, we introduce the following notion.

\begin{definition} We say that a length category $(\mathcal{A},\Theta)$ is {\em $\tau$-tilting finite} if
$$\tors_{\Theta}(\mathcal{A})={\rm s}\text{-}\tors_{\Theta}(\mathcal{A})~\text{and}~\torf_{\Theta}(\mathcal{A})={\rm s}\text{-}\torf_{\Theta}(\mathcal{A}).$$
\end{definition}

\begin{remark}\label{R-5-15} (1) Suppose that $\mathcal{A}$ is Hom-finite and there are only finitely many isomorphism classes of indecomposable objects. We can easily check that $\tors_{\Theta}(\mathcal{A})={\rm s}\text{-}\tors_{\Theta}(\mathcal{A})$ and $\torf_{\Theta}(\mathcal{A})={\rm s}\text{-}\torf_{\Theta}(\mathcal{A})$. This means that $(\mathcal{A},\Theta)$ is $\tau$-tilting finite.

(2) We say a finite dimensional algebra $\Lambda$ is {\em $\tau$-tilting finite} (cf. \cite[Definition 1.1]{De}) if there are only finitely many isomorphism classes of basic  $\tau$-tilting $\Lambda$-modules. It was shown in \cite[Theorem 3.8]{De}) that $\Lambda$ is $\tau$-tilting finite if and only if every torsion class in $\mod \Lambda$ is functorially finite if and only if every torsion-free class in $\mod \Lambda$ is functorially finite. Thus,  our definition coincides with that in \cite[Definition 1.1]{De}.
\end{remark}

When $\mathcal{A}$ is $\tau$-tilting finite, we have the following bijection between support $\tau$-tilting subcategories and support $\tau^{-1}$-tilting subcategories.

\begin{theorem}\label{C-5-15} Suppose that $\mathcal{A}$ has enough $\Theta$-projectives and $\Theta$-injectives. If $\mathcal{A}$ is $\tau$-tilting finite, there is a bijection
$${\rm s}\tau\text{-}{\rm tilt(\mathcal{A})}\longrightarrow{\rm s}\tau^{-1}\text{-}{\rm tilt(\mathcal{A})}$$
given by $\mathcal{S}\mapsto \mathcal{I}(\Fac_{\Theta}(\mathcal{S})^{\perp})$ with the inverse $\mathcal{T}\mapsto \mathcal{P}(^{\perp}\Sub_{\Theta}(\mathcal{T})).$
\end{theorem}
\begin{proof} This follows  from Proposition \ref{L-2-2}, Theorem \ref{main5} and Theorem \ref{main6}.
\end{proof}

\begin{remark}
 In general, the converse of Theorem \ref{C-5-15} is not true. For instance, if  $\mathcal{A}=\mod \Lambda$ for a finite dimensional algebra $\Lambda$, then we have a bijection ${\rm s}\tau\text{-}{\rm tilt~\mathcal{A}}\rightarrow{\rm s}\tau^{-1}\text{-}{\rm tilt~\mathcal{A}}$ (cf. \cite[Section 2.4]{Ad}). But $\mathcal{A}$ is $\tau$-tilting finite if and only if  the set of bricks is finite (cf. \cite[Theorem 4.2]{De}).
\end{remark}

We end this section with an example illustrating some of our results.
\begin{example} Keep the notation used in Example \ref{E-2-22}. Recall that $(\mathcal{A},l_{\mathcal{Y}})$ is a  stable length  category. We know that its Auslander-Reiten quiver is as follows:
\begin{equation}\label{P-1}
\begin{array}{l}
\xymatrix@!=0.5pc{
   {\color{red}|}S_2[-1]\ar[dr] && S_1[-1]\ar[dr] && P_1{\color{red}|}  \\
   & {\color{red}|}I_2[-1]\ar[ur]\ar[dr] && P_2{\color{red}|}\ar[ur] & \\
     && {\color{red}|}S_3{\color{red}|}\ar[ur] &&}
\end{array}
\end{equation}
The $l_{\mathcal{Y}}$-projective (resp. $l_{\mathcal{Y}}$-injective) objects are highlighted by a vertical line to their left (resp. to their right).  By the diagram (\ref{P-1}), it is easy to see that $\mathcal{A}$ has enough  $l_{\mathcal{Y}}$-projectives and $l_{\mathcal{Y}}$-injectives. Moreover, we observe that $(\mathcal{A},l_{\mathcal{Y}})$ is $\tau$-tilting finite by Remark \ref{R-5-15}(1). Therefore, by Corollary \ref{C-5-15}, we have bijections
$${\rm s}\tau\text{-}{\rm tilt(\mathcal{A})}\longleftrightarrow\tor_{l_{\mathcal{Y}}}(\mathcal{A})\longleftrightarrow\torf_{l_{\mathcal{Y}}}(\mathcal{A})\longleftrightarrow{\rm s}\tau^{-1}\text{-}{\rm tilt(\mathcal{A})}.$$
In Table 1, we list  all torsion pairs $(\mathcal{T},\mathcal{F})$, the corresponding support $\tau$-tilting subcategories $\mathcal{P}(\mathcal{T})$ and support $\tau^{-1}$-tilting subcategories $\mathcal{I}(\mathcal{F})$, where the black vertices are $\mathcal{T}$ (resp. $\mathcal{P}(\mathcal{T})$), and the white vertices are $\mathcal{F}$ (resp. $\mathcal{I}(\mathcal{F})$).
\end{example}

\section*{Acknowledgments}
Li Wang is supported by the National Natural Science Foundation of China (Grant No. 12301042) and the  Natural
Science Foundation of Universities of Anhui (No. 2023AH050904). Jiaqun Wei is supported by the National Natural Science Foundation of China (Grant No. 12271249).
Haicheng Zhang is supported by the National Natural Science Foundation of China (No.12271257) and the Natural Science Foundation of Jiangsu Province of China (No.BK20240137).
Panyue Zhou is supported by the National Natural Science Foundation of China (Grant No. 12371034) and the Scientific Research Fund of Hunan Provincial Education Department (Grant No. 24A0221).

\newpage

\begin{center}
\scalebox{2.2}{\includegraphics{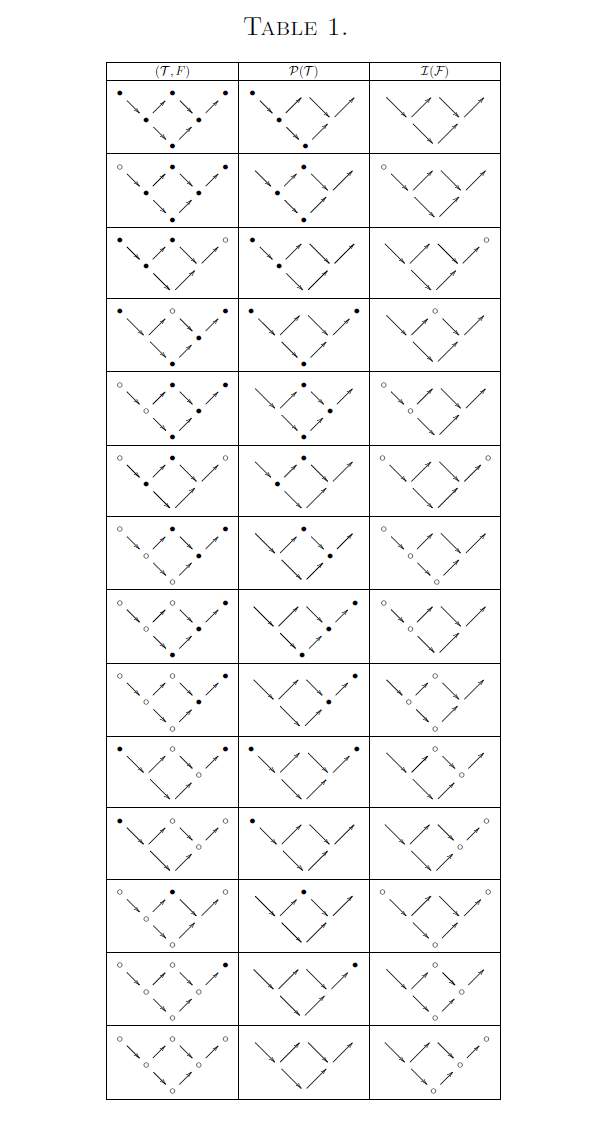}}
\end{center}

\end{document}